\def\year{2021}\relax
\documentclass[letterpaper]{article} 
\usepackage{aaai21}  
\usepackage{times}  
\usepackage{helvet} 
\usepackage{courier}  
\usepackage[hyphens]{url}  
\usepackage{graphicx} 
\usepackage{natbib}  
\usepackage{caption} 
\usepackage{graphicx}
\usepackage{subfigure}
\usepackage{amssymb} 
\usepackage{amsmath} 
\usepackage{amsthm}
\usepackage{amsfonts}       
\usepackage[]{algorithm}
\usepackage[]{algorithmic}
\usepackage{color}
\usepackage{float}
\usepackage[switch]{lineno}

\def\R{\mathbb{R}}

\def\Eps{\mathcal{E}}
\def\E{\mathbb{E}}
\def\L{\mathcal{L}}

\def\D{\mathcal{D}}
\def\O{\mathcal{O}}

\def\B{\mathcal{B}}

\def\P{\mathbb{P}}
\def\1{\mathbf{1}}

\def\X{\mathcal{X}}

\def\eps{\epsilon}

\makeatletter
\newcommand*{\rom}[1]{\expandafter\@slowromancap\romannumeral #1@}
\makeatother

\DeclareMathOperator*{\argmin}{arg\,min}

\newtheorem{definition}{Definition}
\newtheorem{lemma}{Lemma}

\newtheorem{theorem}{Theorem}

\renewenvironment{proof}{{\bfseries Proof.}}{}

\title{Online DR-Submodular Maximization with Stochastic Cumulative Constraints}

\author{

    Prasanna Raut\thanks{Equal contribution} \textsuperscript{\rm 1},
    Omid Sadeghi\footnotemark[1] \textsuperscript{\rm 2},
    Maryam Fazel \textsuperscript{\rm 2}
    
}
\affiliations{

    \textsuperscript{\rm 1}Department of Mechanical Engineering\\
    \textsuperscript{\rm 2} Department of Electrical and Computer Engineering

    University of Washington\\
    Seattle, Washington 98195\\
    \{raut, omids, mfazel\}@uw.edu

}

\begin{document}
\maketitle
\begin{abstract}
  In this paper, we consider online continuous DR-submodular maximization with linear stochastic long-term constraints. Compared to the prior work on online submodular maximization \cite{chenOnlineContinuousSubmodular2018}, our setting introduces the extra complication of stochastic linear constraint functions that are i.i.d. generated at each round. To be precise, at step $t\in\{1,\dots,T\}$, a DR-submodular utility function $f_t(\cdot)$ and a constraint vector $p_t$, i.i.d. generated from an unknown distribution with mean $p$, are revealed after committing to an action $x_t$ and we aim to maximize the overall utility while the expected cumulative resource consumption $\sum_{t=1}^T \langle p,x_t\rangle$ is below a fixed budget $B_T$. Stochastic long-term constraints arise naturally in applications where there is a limited budget or resource available and resource consumption at each step is governed by stochastically time-varying environments. We propose the Online Lagrangian Frank-Wolfe (OLFW) algorithm to solve this class of online problems. We analyze the performance of the OLFW algorithm and we obtain sub-linear regret bounds as well as sub-linear cumulative constraint violation bounds, both in expectation and with high probability.
\end{abstract}

\section{Introduction}

The Online Convex Optimization (OCO) problem has been extensively studied in the literature \citep{hazanIntroductionOnlineConvex2017,10.1561/2200000018,zinkevichOnlineConvexProgramming,orabona2019modern}. In this problem, a sequence of arbitrary convex cost functions $\{f_t(\cdot)\}_{t=1}^T$ are revealed one by one by ``nature" and at each round $t\in[T]$, the decision maker chooses an action $x_t \in \X$, where $\X$ is the fixed domain set, before the corresponding function $f_t(\cdot)$ is revealed. The goal is to minimize the regret defined as \citep{zinkevichOnlineConvexProgramming}
\begin{equation*}
\sum_{t=1}^T f_t(x_t)-\min_{x \in \X}\sum_{t=1}^T f_t(x).
\end{equation*} 
In other words, regret characterizes the difference between the overall cost incurred by the decision maker and that of a fixed benchmark action which has access to all the cost functions $\{f_t\}_{t=1}^T$.

In many applications, however, in addition to maximizing the total reward (minimizing the overall cost), there are restrictions on the sequence of decisions made by the learner that need to be satisfied on average \citep{10.5555/2722129.2722222,10.1145/3164539,10.1145/2600057.2602844,rivera2018online}. 
Therefore, it may be beneficial to sacrifice some of the reward to meet other desired goals or restrictions over the time horizon. 
Such long-term constraints arise naturally in applications with limited budget (resource) availability \citep{doi:10.1287/mnsc.2018.3174,besbes2012blind,ferreira2018online}.\\
As an illustrative example, consider the online ad allocation problem for an advertiser. At each round $t\in[T]$, the advertiser should choose her investment on ads to be placed on $n$ different websites. Beyond the immediate goal of maximizing the overall impressions of the ads, the advertiser needs to balance her total investment against an allotted budget on a daily, monthly or yearly basis \citep{doi:10.1287/mnsc.2018.3174}. However, the cost of ad placement in each round depends on the number of clicks the ads receive, so they are not known ahead of time. Therefore, the advertiser needs to strike the right balance between the total reward and budget used. See Appendix B for a number of other motivating applications that can be naturally cast in our framework.\\
In this paper, we study a new class of online allocation problems with long-term resource constraints where the utility functions are DR-submodular (and not necessarily concave) and the constraint functions are linear with coefficient vectors drawn i.i.d.\ from some unknown underlying distribution. 
The problem has been extensively studied in the convex setting \citep{yuanOnlineConvexOptimization,weiOnlinePrimalDualMirror2019,neelyOnlineConvexOptimization2017,liakopoulosCautiousRegretMinimization}; furthermore, \citet{sadeghiOnlineContinuousDRSubmodular2019} considered a similar framework under the assumption that the linear constraint functions are chosen adversarially. However, \citet{sadeghiOnlineContinuousDRSubmodular2019} do not provide 
any high probability bounds for the regret and constraint violation with random i.i.d.\ linear constraints, and their expected constraint violation bound is worse than ours as well (see Section \ref{related} and the Section \ref{connection} for an overview of related work and comparison of our results with the existing bounds respectively). 
In this paper, we provide the \emph{first} sub-linear bounds for the regret and total budget violation that hold in expectation as well as with high probability. Specifically, our contributions are as follows:
\begin{itemize}
	\item In Section \ref{4.1}, We propose the Online Lagrangian Frank-Wolfe (OLFW) algorithm for this class of online continuous DR-submodular maximization problems with stochastic cumulative constraints. The OLFW algorithm is inspired by the quadratic penalty method in constrained optimization literature \citep{nocedalNumericalOptimization1999} and it generalizes a Frank-Wolfe variant proposed by \citet{chenOnlineContinuousSubmodular2018} for solving online continuous DR-submodular maximization problems to take into account the additional stochastically time-varying linear constraints. Note that this extension is not straightforward and the choice of the penalty function and the update rule for the dual variable are crucial in obtaining bounds for the total budget violation as well as the regret (see Section \ref{4.1} for more details).
	\item We analyze the performance of the OLFW algorithm with high probability and in expectation in Section \ref{4.2} and Section \ref{4.3} respectively and we establish the first sub-linear \emph{expected} and \emph{high probability} bounds on both the \emph{regret} and \emph{total budget violation} of the algorithm.
\end{itemize}
Finally, in Section \ref{exp}, we demonstrate the effectiveness of our proposed algorithm on simulated and real-world problem instances, and compare the performance of the OLFW algorithm with several baseline algorithms. All the missing proofs are provided in the Appendix.
\subsection{Notation}
$[T]$ is used to denote the set $\{1,2,\dots,T\}$. For $u\in\mathbb{R}$, we define $[u]_+ :=\max\{u,0\}$. For a vector $x\in \R^n$, we use $x_i$ to denote the $i$-th entry of $x$. The inner product of two vectors $x,y\in\mathbb{R}^n$ is denoted by either $\langle x, y \rangle$ or $x^T y$. Also, for two vectors $x,y\in \mathbb{R}^n$, $x\preceq y$ implies that $x_i \leq y_i ~\forall i\in[n]$. A function $f:\mathbb{R}^n \to \mathbb{R}$ is called monotone if for all $x,y$ such that $x\preceq y$, $f(x)\leq f(y)$ holds. For a vector $x\in \R^n$, we use $\|x\|$ to denote the Euclidean norm of $x$. The unit ball of the Euclidean norm is denoted by $\B$, i.e., $\B=\{x\in \R^n~|~\|x\|\leq 1\}$. For a convex set $\X$, we will use $\mathcal{P}_{\X}(y)=\argmin_{x\in \X}\|x-y\|$ to denote the projection of $y$ onto set $\X$. The Fenchel conjugate of a function $f:\mathbb{R}^n \to \mathbb{R}$ is defined as $f^*(y)=\sup_{x}(x^Ty-f(x))$.

\section{Preliminaries}
\subsection{DR-submodular functions}\label{dr}
\begin{definition}\label{def:dr}
	A differentiable function $f:\X \rightarrow \mathbb{R}$, $\X\subset \mathbb{R}_+^n$, is called DR-submodular if
	\begin{equation}
	x\succeq y \Rightarrow \nabla f(x) \preceq \nabla f(y)\nonumber.
	\end{equation}
\end{definition}
In other words, $\nabla f$ is element-wise decreasing and satisfies the DR (Diminishing Returns) property.\\
If $f$ is twice differentiable, the DR property is equivalent to the Hessian matrix being element-wise non-positive. Note that for $n=1$, the DR property is equivalent to concavity. However, for $n>1$, concavity corresponds to negative semidefiniteness of the Hessian matrix (which is not equivalent to the Hessian matrix being element-wise non-positive). DR-submodular functions are also known as ``smooth submodular'' in the submodularity literature (e.g., see \citet{vondrakOptimalApproximationSubmodular2008}). \citet{bianGuaranteedNonconvexOptimization2016} showed that a DR-submodular function $f$ is concave along any non-negative and any non-positive direction; that is, if $t\geq 0$ and $v\in \mathbb{R}^n$ satisfies $v\succeq 0$ or $v\preceq 0$, we have
\begin{equation*}
f(x+tv)\leq f(x)+t\langle \nabla f(x),v\rangle.
\end{equation*}
See Appendix A for examples of non-concave DR-submodular functions.

\section{Problem statement}
Consider the following overall offline optimization problem:
\begin{equation}\label{prob}
\begin{array}{ll}
\mbox{maximize}& \sum_{t=1}^T f_t (x_t)\\
\mbox{subject to}& x_t \in \X, ~\forall t\in[T]\\
& \sum_{t=1}^T \langle p, x_t \rangle\leq B_T .
\end{array}
\end{equation}
The online setup is as follows: At each round $t\in[T]$, the algorithm chooses an action $x_t \in \X$, where $\X \subset \R_+^n$ is a fixed, known set. Upon committing to this action, the utility function $f_t:\X \to \R_+$ and a random i.i.d. sample $p_t \sim \D(p, \Sigma)$ are revealed and the algorithm receives a reward of $f_t(x_t)$ while using $\langle p,x_t\rangle$ of its fixed total allotted budget $B_T$. 
The overall goal is to maximize the total obtained reward while satisfying the budget constraint asymptotically (i.e., $\sum_{t=1}^T\langle p, x_t \rangle -B_T$ being sub-linear in $T$). \citet{mahdaviEfficientConstrainedRegret2012} considered a similar setup and performance metric for the special case of linear utility functions.\\
Note that our proposed algorithm can handle multiple linear constraints as well, and similar regret and constraint violation bounds can be derived. However, for ease of notation, we focus on the case with only one linear constraint.\\
We make the following assumptions about our problem framework:\\
\textbf{A1.} The domain $\X \subset \R_+^n$ is a closed, bounded, convex set containing the origin, i.e., $0\in \X$. We denote the diameter of $\X$ with $R$; i.e., $R:=\max_{x,y \in \X}\|y-x\|$.\\
\textbf{A2.} For all $t\in [T]$, the utility function $f_t(\cdot)$ is normalized (i.e., $f_t(0)=0$), monotone, DR-submodular, $\beta_f$-Lipschitz and $L$-smooth. In other words, for all $x,y\in \X$ and $u\in \mathbb{R}^n$ where $u\succeq 0$ or $u\preceq 0$, the following holds:
\begin{align*}
f_t(x+u)-f_t(x)&\geq \langle u, \nabla f_t (x)\rangle -\frac{L}{2}\|u\|^2\\
|f_t(y)-f_t(x)|&\leq \beta_f \|y-x\|.
\end{align*}
\textbf{A3.}\label{A3} For all $t\in [T]$, $p_t\in \R_+^n$ is i.i.d. generated from the distribution $\D$ with bounded support $\beta_p \B \cap \R_+^n$, mean $p\succeq 0$ and covariance matrix $\Sigma$, i.e., $p_t \sim \D(p, \Sigma)$.\\
Let $\beta = \max\{\beta_f,\beta_p\}$.
Under the above assumptions, we have:
\begin{align*}
	F&:=\max_{t\in[T]}\max_{x,y\in\mathcal{X}}|f_t (x)-f_t (y)|\leq \beta R<\infty\\
	G&:=\max_{p' \sim \D(p, \Sigma)}\max_{x\in\mathcal{X}}|\langle p',x\rangle -\frac{B_T}{T}|\leq \beta R-\frac{B_T}{T}<\infty.
\end{align*}
\subsection{Performance metric}
We characterize the performance of our proposed algorithm through bounding the notions of regret and cumulative constraint violation which are defined below:
\begin{definition}
The $(1-\frac{1}{e})$-regret is defined as:
\begin{equation*}
R_T = (1-\frac{1}{e})\max_{x \in \X^*}\sum_{t=1}^T f_t (x)-\sum_{t=1}^T f_t (x_t),
\end{equation*}
where:
\begin{equation*}
\X^* = \{x\in \X:\sum_{t=1}^{T}\langle p,x \rangle \leq B_T\}=\{x\in \X:\langle p,x \rangle \leq \frac{B_T}{T}\}.
\end{equation*}
\end{definition}
The regret metric $R_T$ quantifies the difference of the 
reward obtained by the algorithm and the $(1-\frac{1}{e})$-approximation of the 
reward of the best fixed benchmark action that has access to all the utility functions $f_t ~\forall t\in[T]$, the \emph{mean} $p$  of the linear constraint functions, and satisfies the cumulative budget constraint.
Note that $1-\frac{1}{e}$ is the optimal approximation ratio for offline continuous DR-submodular maximization; in other words, even if all the online input were available beforehand, we could only obtain a $(1-\frac{1}{e})$ fraction of the maximum reward in polynomial time. The $(1-\frac{1}{e})$-regret is commonly used in the online submodular maximization literature  \citep{chenOnlineContinuousSubmodular2018}. 
\begin{definition}
	The cumulative constraint violation is defined as follows:
\begin{equation*}
C_T=\sum_{t=1}^T \langle p_,x_t \rangle -B_T.
\end{equation*}
\end{definition}
Note that since $p_t~\forall t\in[T]$ is i.i.d. drawn from the distribution $\D$ with mean $p$, our cumulative constraint violation metric $C_T$ is defined with respect to the true underlying fixed linear constraint $p$ (as opposed to $p_t$). 

\begin{table*}[t]\centering 
	\begin{tabular}{c|c|c|c|c|c}\label{table1}
		Paper & Cost (utility) & Constraint & Window size & Regret bound & Constraint violation bound \\ \hline
		\citet{yuanOnlineConvexOptimization} & convex & convex (fixed) & --- & $\O (\sqrt{T})$ & $\O (T^{\frac{3}{4}})$ \\ \hline
		\citet{weiOnlinePrimalDualMirror2019} & convex & convex (stochastic) & --- & $\O (\sqrt{T})$ & $\O (\sqrt{T})$ \\ \hline
		\citet{neelyOnlineConvexOptimization2017} & convex & convex (adversarial) & $1$ & $\O (\sqrt{T})$ & $\O (\sqrt{T})$ \\ \hline
		\citet{liakopoulosCautiousRegretMinimization}$^{(a)}$ & convex & convex (adversarial) & $W$ & $\O (\sqrt{T}+\frac{WT}{V})$ & $\O (\sqrt{VT})$ \\ \hline
		\citet{sadeghiOnlineContinuousDRSubmodular2019} & DR-submodular & linear (adversarial) & $W$ & $\O (\sqrt{WT})$ & $\O (W^{\frac{1}{4}}T^{\frac{3}{4}})$ \\ \hline
	\end{tabular}
	\caption{State of the art results for online problems with cumulative constraints in various settings. Note that in (a), $V\in (W,T)$ is a tunable parameter.}
\end{table*}

\subsection{Related work}\label{related}
Consider the following general framework of online problems with long-term constraints: At round $t\in[T]$, the player chooses $x_t \in \mathcal{X}$. Then, cost (utility) function $f_t: \mathcal{X}\rightarrow \R$ ($\mathcal{X}$ is a fixed convex set) and constraint function $g_t: \mathcal{X}\rightarrow \R$ are revealed, the player incurs a loss (obtains a reward) of $f_t(x_t)$ and uses the amount $g_t (x_t)$ of her budget (with the long-term constraint $\sum_{t=1}^Tg_t(x_t)\leq 0$). This problem has been extensively studied under various assumptions where the cost (utility) functions are adversarially chosen and are assumed to be linear, convex or DR-submodular and the constraint functions are linear or convex and are either fixed (i.e., $g_t(\cdot)=g(\cdot)~\forall t\in [T]$), stochastic and i.i.d drawn from some unknown distribution, or adversarial. For the setting with adversarial utility and constraint functions, \citet{mannorOnlineLearningSample} provided a simple counterexample to show that regardless of the decisions of the algorithm, it is impossible to guarantee sub-linear regret against the benchmark action while the overall budget violation is sub-linear. Therefore, prior works in this setting have further restricted the fixed comparator action to be chosen from $\X_W = \{x\in \X:\sum_{\tau=t}^{t+W-1}g_{\tau}(x)\leq 0,~1\leq t\leq T-W+1\}$. In other words, in addition to merely satisfying the overall cumulative constraint (which corresponds to the $W=T$ case), the benchmark action is required to satisfy the budget constraint proportionally over any window of length $W$. On the other hand, for fixed or stochastic constraint functions, sub-linear regret and constraint violation bounds have been derived in the literature. A summary of the state of the art results for online problems with long-term constraints is provided in Table $1$.

\section{Online Lagrangian Frank-Wolfe (OLFW) algorithm}
In this section of the paper, we first introduce our proposed algorithm, namely the Online Lagrangian Frank-Wolfe (OLFW) algorithm, in Section \ref{4.1} and subsequently, we analyze the performance of the algorithm with high probability and in expectation in Section \ref{4.2} and Section \ref{4.3} respectively. 
\subsection{Algorithm}\label{4.1}
The Online Lagrangian Frank-Wolfe (OLFW) algorithm is presented in Algorithm \ref{olfw}. First, note that for all $t\in[T]$, $x_t=\frac{1}{K}\sum_{k=1}^K v_t^{(k)}$ is the average of vectors in the convex domain $\X$ and hence, $x_t \in \X$. The intuition for using $K$ online maximization subroutines to update $x_t$ is the Frank-Wolfe variant proposed in \citet{bianGuaranteedNonconvexOptimization2016} to obtain the optimal approximation guarantee of $1-\frac{1}{e}$ for solving the offline DR-submodular maximization problem without the additional linear constraints. To be more precise, consider the first iteration $t=1$ of our online setting (ignoring the linear cumulative constraints) and the corresponding DR-submodular utility function $f_1(\cdot)$ arriving at this step. Note that $f_1$ is not revealed until the algorithm commits to an action $x_1 \in \mathcal{X}$. If we were in the offline setting, we could use the mentioned Frank-Wolfe variant of \citet{bianGuaranteedNonconvexOptimization2016}, run it for $K$ iterations and maximize $f_1$ over $\mathcal{X}$. Starting from $x_1^{(1)}=0$, for all $k\in[K]$, we would find a vector $v_1^{(k)}$ that maximizes $\langle x,\nabla f_1(x_1^{(k)}) \rangle$ over $x\in \mathcal{X}$, perform the update $x_1^{(k+1)}=x_1^{(k)}+\frac{1}{K} v_1^{(k)}$ and derive $x_1=x_1^{(K+1)}$ as the output. However, in the online setting, the utility function $f_1$ is not available before committing to the action $x_1$. Therefore, for each $k\in[K]$, we instead use a separate instance of a no-regret online linear maximization algorithm to obtain $v_1^{(k)}$. We repeat the same process for the subsequent utility functions $\{f_t\}_{t>1}$. This intuition was first provided in \citet{chenOnlineContinuousSubmodular2018} and they managed to obtain an $\mathcal{O}(\sqrt{T})$ regret bound for the unconstrained online monotone submodular maximization problem.\\
Our choice of Lagrangian function is inspired by the quadratic penalty method in constrained optimization \citep{nocedalNumericalOptimization1999}. The penalized formulation of the overall optimization problem $(\ref{prob})$ with quadratic penalty function could be written as follows:
\begin{align*}
&\max_{x_t}~~\sum_{t=1}^T f_t(x_t)-\frac{1}{2\delta \mu}\big(\sum_{t=1}^T \langle p,x_t \rangle -B_T\big)^2\\
&\text{subject to}~~x_t \in \X~\forall t\in[T].
\end{align*}
Considering that the Fenchel conjugate of the function $h(\cdot)=\frac{1}{2\delta \mu}(\cdot)^2$ is $h^*(\cdot)=\frac{\delta \mu}{2}(\cdot)^2$, we can write the above problem in the following equivalent form:
\begin{align*}
\max_{x_t}\min_{\lambda}&~~\sum_{t=1}^T f_t(x_t)-\lambda\big(\sum_{t=1}^T \langle p,x_t \rangle -B_T\big)+\frac{\delta \mu}{2}\lambda^2\\
\text{subject to}&~~x_t \in \X~\forall t\in[T].
\end{align*}
Therefore, the corresponding Lagrangian function at round $t\in [T]$ is $\L_t(x,\lambda)=f_t(x)-\lambda (\langle p,x\rangle -\frac{B_T}{T})+\frac{\delta \mu}{2}\lambda^2$. However, $p$ is unknown to the online algorithm. Therefore, we alternatively use $\widehat{p}_t := \frac{1}{t-1}\sum_{s=1}^{t-1}p_s$ instead of $p$ in the Lagrangian function. Note that $\widehat{p}_t$ is the empirical estimation of $p$ at round $t$.
\begin{algorithm}[t]
	\caption{Online Lagrangian Frank-Wolfe (OLFW)}
	\begin{algorithmic}\label{olfw}
		\STATE \textbf{Input:} $\X$ is the constraint set, $T$ is the horizon, $\mu>0$, $\delta>0$, $\{\gamma_t\}_{t=1}^{T}$ and $K$.
		\STATE \textbf{Output:} $\{x_t :1\leq t\leq T\}$.
		\STATE Initialize $K$ instances $\{\Eps_k\}_{k\in[K]}$ of Online Gradient Ascent with step size $\mu$ for online maximization of linear functions over $\mathcal{X}$.
		\FOR{$t=1$ {\bfseries to} $T$}
		\STATE $x_t^{(1)}=0$.
		\FOR{$k=1$ {\bfseries to} $K$}
		\STATE Let $v_t^{(k)}$ be the output of oracle $\mathcal{E}_k$ from round $t-1$.
		\STATE $x_t^{(k+1)}=x_t^{(k)}+\frac{1}{K} v_t^{(k)}$.
		\ENDFOR
		\STATE Set $x_t =x_t^{(K+1)}$.
		\STATE Let $\widehat{p}_t := \frac{1}{t-1}\sum_{s=1}^{t-1}p_s$ for $t >1$.
		\STATE Let \[ \widetilde{g}_t(\cdot)= 
		\begin{cases}
		\langle \widehat{p}_t, \cdot \rangle -\frac{B_T}{T} &\quad\text{expectation analysis (\rom{1})}\\
		\langle \widehat{p}_t, \cdot \rangle -\frac{B_T}{T} - \gamma_t &\quad\text{high probability analysis (\rom{2})}\\
		\end{cases}.
		\] 
		\STATE Set $\lambda_{t}= \frac{[\widetilde{g}_t(x_t)]_+}{\delta \mu}$ for $t > 1$ and 0 otherwise.
		\STATE Play $x_t$ and observe the Lagrangian function $\L_t(x_t,\lambda_t)=f_t(x_t)-\lambda_t \widetilde{g}_t(x_t)+\frac{\delta \mu}{2}\lambda_t^2$.
		\FOR{$k=1$ {\bfseries to} $K$}
		\STATE Feedback $\langle v_t^{(k)},\nabla_x \L_t (x_t^{(k)},\lambda_t)\rangle$ as the payoff to be received by $\mathcal{E}_k$.
		\ENDFOR
		\ENDFOR
	\end{algorithmic}
\end{algorithm}

We first provide a lemma which is central to obtaining the regret and constraint violation bounds both in expectation and with high probability.
\begin{lemma}\label{lem4.1}
	Let $x\in \X$ be a fixed vector. In the OLFW algorithm, set $\delta = \beta^2$. We then have:
	\begin{align}
	\sum_{t=1}^{T}\big((1-\frac{1}{e})f_{t}(x)-f_{t}(x_{t})\big) & \leq \frac{LR^2T}{2K}+\frac{R^2}{\mu}+\beta^2 \mu T\nonumber\\
	&+\sum_{t=1}^{T}\lambda_t \widetilde{g}_t(x) \label{reg44}.
	\end{align}
\end{lemma}
\subsection{Performance analysis with high probability}\label{4.2}
In order to analyze the performance of the OLFW algorithm with high probability, the following lemmas detailing the concentration inequalities for the stochastic linear constraints are used.
\begin{lemma}\label{cor4.4}
	The following holds with probability at least $1-\eps$:
	\begin{align*}
	\sum_{t=2}^{T}\|\widehat{p}_t - p\| \leq C\sigma \sqrt{T\log\big(\frac{2nT}{\eps}\big)}.
	\end{align*}
\end{lemma}
\begin{lemma}\label{lem4.5}
	Let $x\in \X$ be fixed. Define $\widehat{g}_t(x) := \langle\widehat{p}_t, x\rangle - \frac{B_T}{T}$ and $g(x):= \langle p, x \rangle - \frac{B_T}{T}$. For a fixed $t \in \{2,3,\ldots,T\}$ and $\{\gamma_t := \sqrt{\frac{2G^2\log(\frac{2T}{\eps})}{t}}\}_{t=2}^{T}$, $|\widehat{g}_t(x) - g(x)| \leq \gamma_t$ holds with probability at least $1-\frac{\eps}{T}$.
\end{lemma}
\begin{proof}
	First, note that $\E[\widehat{g}_t(x)] = \E[\langle \widehat{p}_t, x\rangle - \frac{B_T}{T}] = \langle p,x \rangle -\frac{B_T}{T}= g(x)$. If $y_t = g_t(x)$ is a random variable, then by assumption, $y_t \in [-G, G]$ holds for each $t$, i.e., $y_t$ is a bounded random variable. Therefore we can apply Hoeffding's inequality and get $\P\{|\widehat{g}_t(x) - g(x)| > \gamma_t\} \leq 2\exp(-\frac{t\gamma_t^2}{2G^2})$. Substituting the value of $\gamma_t$ in the right hand side, we get that $\P\{|\widehat{g}_t(x) - g(x)| > \gamma_t\} \leq \frac{\eps}{T}$. The result follows immediately.\qed
\end{proof}


Now, we have all the machinery to obtain the high probability performance bounds of the OLFW algorithm. 
\begin{theorem}\label{thm4.6}
\textbf{(High probability regret bound)} Let $\eps \in (0,1)$ be given. Set $\mu= \frac{R}{\beta\sqrt{T}}$, $K = \sqrt{T}$, $\delta = \beta^2$ and $\{\gamma_t\}_{t=2}^{T}$ be chosen according to  Lemma \ref{lem4.5}. Then, the OLFW algorithm with update (\rom{2}) for $\widetilde{g}_t(\cdot)$ obtains the following regret bound with probability at least $1-\eps$.
\begin{align}
R_T \leq \big(\frac{LR^2}{2} + 2R\beta\big)\sqrt{T}\nonumber.
\end{align}
\end{theorem}
\begin{proof}
We begin from Lemma \ref{lem4.1}. Substitute the benchmark $x = x^*$ as the fixed vector in ($\ref{reg44}$) and the constants as given in the hypothesis. We get: $R_T \leq (\frac{LR^2}{2} + 2R\beta)\sqrt{T} + \sum_{t=1}^{T}\lambda_t \widetilde{g}_t(x^*)$. Now let us bound $\sum_{t=1}^{T}\lambda_t \widetilde{g}_t(x^*)$. From Lemma \ref{lem4.5}, we have that with probability at least $1 - \frac{\eps}{T}$, $\widehat{g}_t(x^*) - \gamma_t \leq g(x^*)$ holds, i.e., $\widetilde{g}_t(x^*) \leq g(x^*)$. Also, $g(x^*) \leq 0$ holds according to the definition of the benchmark action. Therefore, we have: $\widetilde{g}_t(x^*) \leq 0$. As $\lambda_t \geq 0$, $\lambda_t\widetilde{g}_t(x^*) \leq 0$ holds. Now, taking union bound over all $t\in[T]$, we have with probability at least $1 - \eps$ that $\sum_{t=1}^{T}\lambda_t \widetilde{g}_t(x^*) \leq 0$.
The result follows immediately. \qed
\end{proof}\\
We will use the following lemma to get performance bounds for the constraint violation.
\begin{lemma}\label{lem4.7}
Let $\{\gamma_t\}_{t=2}^{T}$ be defined as in Lemma \ref{lem4.5}, then the following holds.
\begin{align*}
C_T \leq \sum_{t= 1}^{T}[\widetilde{g}_t(x_t)]_+ + R\sum_{t=1}^{T}\|\widehat{p}_t-p\| + \sum_{t=2}^{T}\gamma_t,
\end{align*}
where $\widetilde{g}_t(\cdot)$ is derived using update (\rom{2}).
\end{lemma} 
\begin{theorem}\label{thm4.8}
(\textbf{High probability constraint violation bound}) Let $\eps \in (0,1)$ be given. Set $\mu = \frac{R}{\beta\sqrt{T}}$, $K = \sqrt{T}$, $\delta = \beta^2$ and $\{\gamma_t\}_{t=2}^{T}$ be chosen according to Lemma \ref{lem4.5}. Then the following holds with probability at least $1-\eps$ for the OLFW algorithm with update rule (\rom{2}).
\begin{align*}
C_T& \leq \sqrt{2G^2T\log(\frac{2T}{\eps})} + CR\sigma\sqrt{T\log(\frac{2nT}{\eps})}  \nonumber\\
&+ \frac{T}{B_T}R\beta F\sqrt{T}+\frac{TR\beta}{B_T}(\frac{LR^2}{2}+2R\beta) + R\beta.
\end{align*}
So, we obtain $\tilde{\O}(\sqrt{T})$ constraint violation bound with high probability.
\end{theorem}
\begin{proof}
We begin with Lemma \ref{lem4.1} again but now substitute $x= 0$ as the fixed vector in ($\ref{reg44}$).
\begin{align}
\frac{B_T}{T}\sum_{t=1}^T \lambda_t + \sum_{t=1}^{T}\lambda_t \gamma_t&\leq \underbrace{\sum_{t=1}^{T}f_{t}(x_{t})}_{\leq FT} +\frac{LR^2T}{2K}+\frac{R^2}{\mu} +\beta^2 \mu T.\label{regOT}
\end{align}
Rearranging and substituting the values of input parameters as given in the hypothesis, we get: 
\begin{align*}
\sum_{t=1}^T [\widetilde{g}_t(x_t)]_+  + \frac{T\delta \mu}{B_T}\sum_{t=1}^{T}\lambda_t \gamma_t &\leq \frac{T}{B_T}R\beta F\sqrt{T}\\
& +\frac{TR\beta}{B_T}(\frac{LR^2}{2}+2R\beta). 
\end{align*}
Both terms in the left hand side of the above equation are positive. Thus, we can drop the second term. We have: 
\begin{align*}
\sum_{t=1}^T [\widetilde{g}_t(x_t)]_+  &\leq \frac{T}{B_T}R\beta F\sqrt{T}+\frac{TR\beta}{B_T}(\frac{LR^2}{2}+2R\beta).
\end{align*}
Combining Lemma \ref{lem4.7} and the equation above, we obtain:
\begin{align}
C_T & \leq \frac{T}{B_T}R\beta F\sqrt{T} \nonumber +\frac{TR\beta}{B_T}(\frac{LR^2}{2}+2R\beta) \nonumber\\
& + R\sum_{t=1}^{T}\|\widehat{p}_t-p\| + \sum_{t=2}^{T}\gamma_t \nonumber.
\end{align}
Therefore, we can conclude:
\begin{align}
C_T& \leq \frac{T}{B_T}R\beta F\sqrt{T} \nonumber +\frac{TR\beta}{B_T}(\frac{LR^2}{2}+2R\beta) \nonumber\\
&+ \sqrt{2G^2T\log(\frac{2T}{\eps})} + \underbrace{R\sum_{t=1}^{T}\|\widehat{p}_t-p\|}_{\text{(A)}} \nonumber,
\end{align}
where the last inequality follows from summing $\gamma_t$'s. Now, Lemma \ref{cor4.4} tells us that $\text{(A)} \leq R\beta + CR\sigma\sqrt{T\log(\frac{2nT}{\eps})}$ holds with probability at least $1- \eps$. Thus, we get the result. \qed
\end{proof}\\
Theorem \ref{thm4.6} and Theorem \ref{thm4.8} are indeed the first high probability bounds obtained for the online DR-submodular maximization problem with stochastic cumulative constraints. Note that the $\O(\sqrt{T})$ regret bound obtained in Theorem \ref{thm4.6} is known to be optimal. 
\subsection{Performance analysis in expectation}\label{4.3}
We first provide a simple lemma that will be used throughout the analysis in expectation.
\begin{lemma}\label{lem4.9}
	For $t > 1$, we have:
	\begin{equation*}
	\E\|\widehat{p}_t-p\|^2 = \frac{Tr(\Sigma)}{t-1},
	\end{equation*}
	where $Tr(\Sigma)$ denotes the trace of the covariance matrix $\Sigma$.
\end{lemma}
Now, we present the main performance bounds in expectation, namely the expected regret bound and the expected cumulative constraint violation bound. In the Appendix, we have also considered the case where we only have access to unbiased stochastic gradient estimates of the utility functions $\{f_t\}_{t=1}^T$ and exact gradient computation is not possible. For this setting, we modify the OLFW algorithm through incorporating the variance reduction technique introduced by \citet{pmlr-v80-chen18c} and we obtain similar regret and constraint violation bounds in expectation for the modified algorithm.
\begin{theorem}\label{thm4.10}
	(\textbf{Expected regret bound})
	The regret bound of the OLFW algorithm with update rule (\rom{1}) is the following:
	\begin{align*}
	\E[R_T]\leq \tilde{\O}(T^{\frac{3}{4}}).
	\end{align*}
\end{theorem}
\begin{proof}
	We first observe from ($\ref{regOT}$) that $\sum_{t=1}^{T}\lambda_t \leq \O(T)$. Now, substitute $x = x^*$, the benchmark, in $(\ref{reg44})$ and take expectation on both sides to obtain:
	\begin{align*}
	\E[R_T]&\leq \frac{LR^2T}{2K}+\frac{R^2}{\mu}+\beta^2 \mu T \nonumber \\ 
	& +\E[\sum_{t=1}^{T}\lambda_t(\widehat{g}_t(x^*) - g(x^*))] + \sum_{t=1}^{T}\lambda_t g(x^*)\nonumber \\
	& \leq \frac{LR^2T}{2K}+\frac{R^2}{\mu}+\beta^2 \mu T + \E\underbrace{[\sum_{t=1}^{T}\lambda_{t}(\widehat{g}_{t}(x^*) - g(x^*))]}_{\text{(B)}}. 
	\end{align*}
	Now we bound (B) as follows:
	\begin{align*}
	\text{(B)} & = \sum_{t=1}^{T}\lambda_{t}(\widehat{g}_{t}(x^*) - g(x^*)) \nonumber \\
	& \leq  \sqrt{\sum_{t=1}^{T}\lambda_t^2}\sqrt{\sum_{t=1}^{T}(\widehat{g}_{t}(x^*) - g(x^*))^2} \nonumber \\
	& = \sqrt{\|\lambda\|^2}\sqrt{\sum_{t=1}^{T}(\langle\widehat{p}_t - p, x^* \rangle)^2} \nonumber \\
	& \leq \|\lambda \| \sqrt{\sum_{t=1}^{T}\|\widehat{p}_t - p \|^2R^2} \nonumber \\
	& = R\|\lambda\| \sqrt{\sum_{t=1}^{T}\|\widehat{p}_t-p\|^2}.
	\end{align*}
	Both the inequalities above are obtained using Cauchy-Schwarz inequality, where $\lambda := [\lambda_1, \lambda_2, \ldots, \lambda_T]^T$.
	
	Using the Cauchy-Schwarz inequality again, we have $\|\lambda\| \leq \sqrt{\|\lambda\|_1 \|\lambda\|_{\infty}}$. Thus, we obtain $\|\lambda\| \leq \sqrt{(\sum_{t=1}^{T}\lambda_t)(\frac{G}{\delta \mu})} \leq \O(T^{3/4})$. Therefore, the following holds:
	\begin{align}
	\|\lambda\| \leq \O(T^{3/4}) .\label{reg44_9d}
	\end{align}
	Using Jensen's inequality, we have \[\E\sqrt{\sum_{t=1}^{T}\|\widehat{p}_t - p\|^2} \leq \sqrt{\sum_{t=1}^{T}\E\|\widehat{p}_t-p\|^2.}\] We can use Lemma \ref{lem4.9} and write:
	\begin{align}
	\E\sqrt{\sum_{t=1}^{T}\|\widehat{p}_t - p\|^2} \leq \sqrt{Tr(\Sigma) \log(T)}\label{reg44_9e}.
	\end{align}
	Thus, combining ($\ref{reg44_9d}$) and ($\ref{reg44_9e}$), we obtain:
	\begin{equation*}
	\E[(\text{B})] \leq \O(T^{3/4}\sqrt{Tr(\Sigma)\log(T)}).
	\end{equation*}
	The result thus follows.
	\qed
\end{proof}
\paragraph{Remark.} The main challenge in bounding $R_T$ in expectation is the fact that in our algorithm, the choice of $\lambda_t$ is dependent on $\widehat{p}_t$ and thus, we cannot use 
\begin{equation*}
\E [\lambda_t \widehat{g}_t(x^*)]=\E [\lambda_t\E[\widehat{g}_t (x^*)]] = \E [\lambda_tg(x^*)]\leq 0
\end{equation*}
and this term is indeed the dominating term in the regret bound. However, as we saw earlier, we do not encounter this problem in the high probability setting due to subtracting $\gamma_t$ from all the constraint functions and using the concentration inequalities, and thus we were able to obtain $\mathcal{O}(\sqrt{T})$ high probability regret bound.
\begin{theorem}\label{thm4.11}
(\textbf{Expected cumulative constraint violation bound}) For the OLFW algorithm with update rule (\rom{1}), we have:
\begin{align}
\E[C_T] \leq & \frac{T}{B_T}R\beta F\sqrt{T} + R\sqrt{Tr(\Sigma)}\sqrt{T} \nonumber \\
&+\frac{TR\beta}{B_T}(\frac{LR^2}{2}+2R\beta) + R\beta. \nonumber
\end{align}
Therefore, $\E[C_T]\leq \tilde{\O}(\sqrt{T})$ holds.
\end{theorem}
Theorem \ref{thm4.10} and Theorem \ref{thm4.11} provide the first sub-linear expectation bounds on the regret and cumulative constraint violation of the online DR-submodular maximization problem with stochastic cumulative constraints. 

\subsection{Relation with previous results}\label{connection}
\begin{itemize}
	\item \citet{yuanOnlineConvexOptimization} studied a similar problem in the convex setting where all the constraint functions are deterministic and given offline and they obtained $\mathcal{O}(\sqrt{T})$ regret bound and $\mathcal{O}(T^{\frac{3}{4}})$ constraint violation bound. On the other hand, applying our OLFW algorithm with update rule (\rom{1}) to the online DR-submodular maximization problem subject to deterministic linear constraints, we obtain $\mathcal{O}(\sqrt{T})$ regret and constraint violation bounds simultaneously. Note that in this setting, $\Sigma = 0$ and thus, the dominating $\mathcal{O}(T^{\frac{3}{4}})$ term in the expected regret bound of the OLFW algorithm vanishes.
	\item \citet{jenattonOnlineOptimizationRegret2016} considered a related problem with concave utility functions and linear constraint functions where at each round, the reward functions arrive before committing to an action (i.e., the 1-lookahead setting) and long-term constraints are penalized through a penalty function in the objective. They obtained sub-linear bounds for the dynamic regret in their setting. On the contrary, in our framework, we deal with the extra complication that the utility function at each step is DR-submodular (and generally non-concave) and it is revealed after committing to an action. 
	\item \citet{mahdaviEfficientConstrainedRegret2012} considered the exact same framework as ours in the special case where the utility functions are linear and they obtained $\mathcal{O}(\sqrt{T})$ regret bound and $\mathcal{O}(T^{\frac{3}{4}})$ constraint violation bound in both expectation and high probability settings. Note that our OLFW algorithm obtains improved $\mathcal{O}(\sqrt{T})$ constraint violation bounds in expectation and with high probability for the more general setting of DR-submodular utility functions. 
	\item \citet{sadeghiOnlineContinuousDRSubmodular2019} considered the online DR-submodular maximization problem subject to linear constraints in the adversarial setting where the constraint functions are chosen arbitrarily. Their proposed OSPHG algorithm in the setting with window length $W=1$ could be adapted to obtain $\O(\sqrt{T})$ regret and $\O(T^{\frac{3}{4}})$ constraint violation bounds in expectation. However, the OSPHG algorithm fails to provide \emph{any} bounds for the high probability setting. On the other hand, our algorithm uses the current estimate of the cost vector to exploit the stochastic nature of the constraints. Furthermore, through using the update rule (\rom{2}) in the OLFW algorithm, we are able to guarantee sub-linear bounds in the high probability setting as well. 
\end{itemize}

\section{Numerical results}\label{exp}
We conduct numerical experiments, over simulated and real-world datasets in the following.\\
\textbf{Joke Recommendation} We look at the problem of DR-submodular function maximization over the \textit{Jester} dataset\footnote{http://eigentaste.berkeley.edu/dataset/}. We consider a fraction of the dataset where there are 100 jokes and user ratings from 10000 users are available for these jokes. The ratings take values in $[-10, 10]$, we re-scale them to be in $[0, 10]$. Let $R_{u,j}$ be the rating of user $u$ for joke $j$. As some of the user ratings are missing in the dataset, we set such ratings to be $5$. In the online setting, a user arrives and we have to recommend at most $M=15$ jokes to her. The utility function for each round $t\in[T]$ is of the form $f_t(x) = \sum_{i=1}^{100}R_{u_t, i}^t x_i+\sum_{i,j:i\neq j}\theta_{ij}x_i x_j$, where $u_t$ is the user being served in the current round. $\{\theta_{ij}\}_{i\neq j}$ are chosen such that the function is monotone. These DR-submodular utility functions capture the overall impression of the displayed jokes on the user. There is a limited total time (denoted by $B_T = 1.5T$) available to recommend the jokes to the users. For all $i \in [n], p_i$ denotes the average time it takes to read joke $i$. As some jokes are relatively longer, we do not want the user to spend more time on jokes which do not lead to larger utility. The linear budget functions are chosen randomly with entries uniformly drawn from $[0.03, 0.35]$. We compare the performance of our algorithm against the following strategies:
\begin{itemize}
    \item \textit{Uniform}: At every round, we assign 15 randomly chosen jokes to the user.
    \item \textit{Greedy}: We deploy an exploration-exploitation strategy where with probability 0.1, we randomly assign 15 jokes and with probability 0.9, we present the top 15  jokes based on the ratings observed so far.
    \item \textit{Meta-FW} \citep{chenOnlineContinuousSubmodular2018}: This corresponds to solving the unconstrained DR-submodular maximization problem (i.e., ignoring the budget constraints).
    \item \textit{Budget-Cautious}: At each round, we assign 15 jokes which have the lowest average budget consumption observed so far.
\end{itemize}
The results are presented in Figure \ref{Fig1}. As it can be seen in the plots, our OLFW algorithm obtains a reasonable utility while approximately satisfying the budget constraint as well.

\textbf{Indefinite quadratic functions.} We choose $\X = \{x\in\R^2: 0\preceq x \preceq 1\}$ and for each $t\in[T]$, we generate quadratic functions of the form $f_t(x) = \frac{1}{2}x^T H_tx+h_t^T x$ where $H_t \in \R^{2\times 2}$ is a random matrix whose entries are chosen uniformly from $[-1,0]$. We let $h_t = -H_t^T\1 $ to ensure the monotonocity of the objective. We let $T = 1000$. At each round, we randomly generate linear budget functions whose entries are chosen uniformly from $[0.5, 2.5]$ and the mean vector is $p = [1, 2]^T$. Also, we set the total budget to be $B_T = 2T$. We run the OLFW algorithm 10 times and take the respective averages for the cumulative utility and total remaining budget. We vary $\delta$, the parameter of the penalty function, in the range $[0.1, 1000]$ and plot the trade-off curve $\big(\text{i.e.}, \sum_{t=1}^{1000}f_t(x_t)$ versus $B_T - \sum_{t=1}^{1000}\langle p, x_t \rangle\big)$ for 100 chosen values of $\delta$ in Figure \ref{Fig2}. In this example, our choice of $\delta$ in the OLFW algorithm, highlighted in the plot, achieves the highest possible cumulative utility while satisfying the total budget constraint.

\textbf{Log-determinant functions.}
We choose $\X = \{x \in \R^{10}: 0 \preceq x \preceq 1\}$ and for each $t \in [T]$, we generate log-determinant functions of the form $f_t(x) = \log \det \big({\rm diag}(x)(L_t-I) + I\big)$, where each $L_t$ is a random positive definite matrix with eigenvalues falling in the range $[2,3]$. The choice of eigenvalues ensures the monotonocity of the function. Let $T= 4900$. At each round, we generate linear budget functions whose entries are chosen uniformly from the range $[0.3, 5.7]$. We run the OLFW algorithm for different choices of the step size $\mu$ and plot the cumulative utility and the total budget violation in Figure \ref{Fig3}. Our OLFW algorithm chooses $\mu$ such that the overall utility and cumulative budget consumption are balanced.
\section{Conclusion and future work}
In this work, we studied online continuous DR-submodular maximization with stochastic linear cumulative constraints. We proposed the Online Lagrangian Frank-Wolfe (OLFW) algorithm to solve this problem and we obtained the first sub-linear bounds, both in expectation and with high probability, for the regret and constraint violation of this algorithm. The current work could be further extended in a number of interesting directions. First, it is yet to be seen whether the online DR-submodular maximization setting could handle general, stochastic or adversarial, convex long-term constraints. Furthermore, it is interesting to see whether it is possible to improve the expected regret bound to match the $\O (\sqrt{T})$ high probability regret bound. Finally, studying this problem under bandit feedback (as opposed to the full information setting considered in this paper) is left to future work.
\begin{figure}[H]
    \centering
	\includegraphics[scale=0.4]{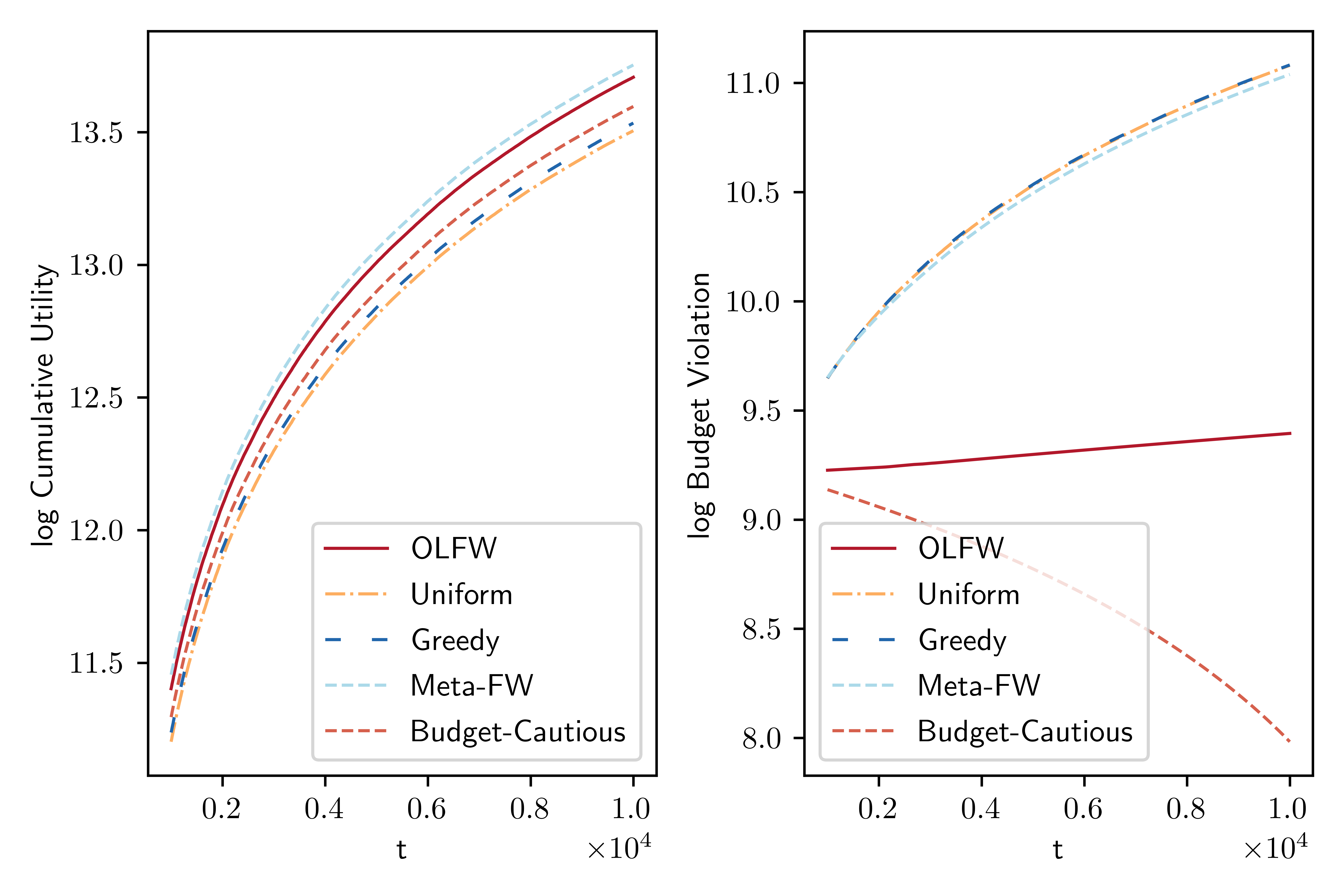}
	\caption{Comparison of the overall utility $\sum_{t=1}^{T}f_t(x_t)$ and cumulative budget violation $\sum_{t=1}^T\langle p, x_t\rangle - B_T$ for the Jester dataset.} 
	\label{Fig1}
\end{figure}
\begin{figure}[H]
    \centering
	\includegraphics[scale=0.4]{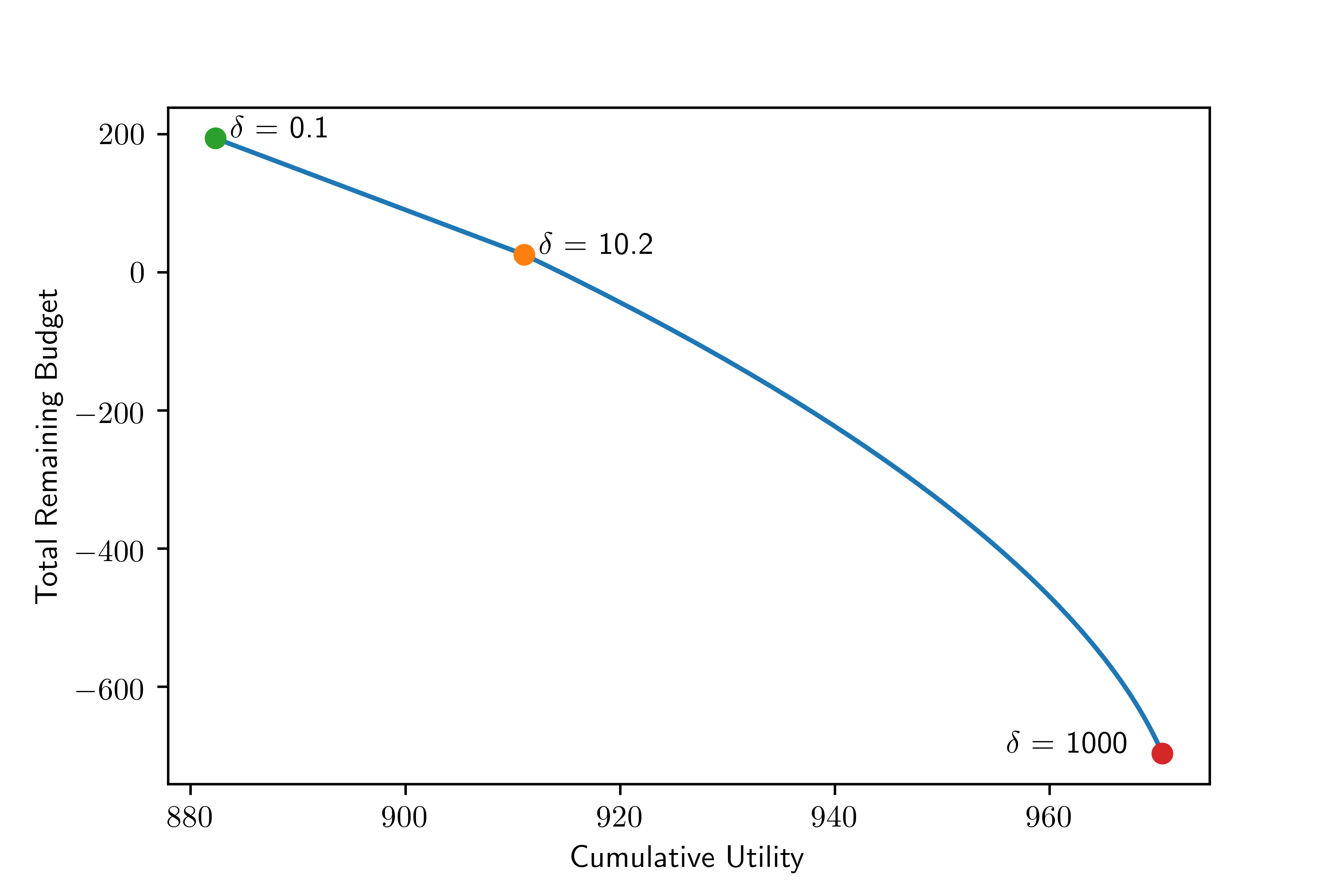}
	\caption{Trade-off between the overall utility $\sum_{t=1}^{T}f_t(x_t)$ and the total remaining budget $B_T-\sum_{t=1}^T\langle p, x_t\rangle$ of quadratic functions for different choices of parameter $\delta$. $\delta = 10.2$ is our choice of the penalty parameter.}
	\label{Fig2}
\end{figure}
\begin{figure}[H]
    \centering
	\includegraphics[scale=0.4]{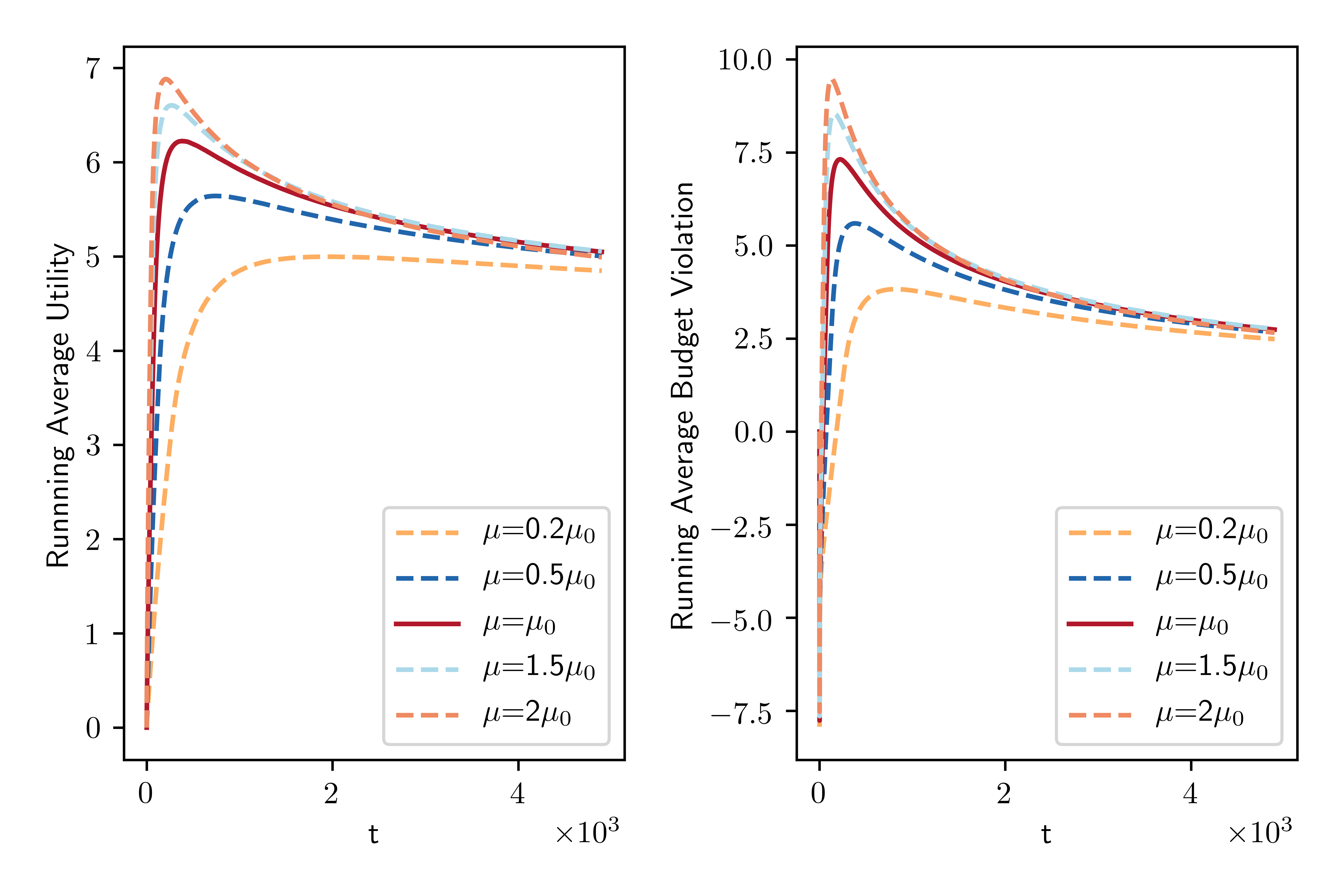}
	\caption{Running average of cumulative utility $\frac{1}{t}\sum_{\tau=1}^{t}f_{\tau}(x_{\tau})$ and running average of budget violation $\frac{1}{t}\sum_{\tau=1}^{t}(\langle p, x_{\tau}\rangle - \frac{B_T}{T})$ for different choices of the step size $\mu$, where $\mu_0 := \frac{R}{\beta\sqrt{T}}$ is our choice of step size in the OLFW algorithm.}
	\label{Fig3}
\end{figure}
\newpage
\section*{Acknowledgements}
This work was supported in part by the following grants:
NSF TRIPODS grant 1740551, 
DARPA Lagrange grant FA8650-18-2-7836, 
ONR MURI grant N0014-16-1-2710.

\bibliography{references}
%
\clearpage
\appendix
\section*{Appendix}

\section{Examples of non-concave DR-submodular functions}\label{examples}

\paragraph{\textbf{Multilinear extension of discrete submodular set functions.}} A discrete function $F:\{0,1\}^V\rightarrow \mathbb{R}$ is submodular if for all $j\in V$ and for all sets $A\subseteq B\subseteq V\setminus\{j\}$, the following holds:
\begin{equation*}
F(A\cup\{j\})-F(A)\geq F(B\cup\{j\})-F(B).
\end{equation*} 
The multilinear extension $f:[0,1]^V \rightarrow \mathbb{R}$ of $F$ is defined as \citep{calinescuMaximizingMonotoneSubmodular2011}
\begin{equation*}
f(x)=\sum_{S\subset V}F(S)\prod_{i\in S}x_i \prod_{j\notin S}(1-x_j)=\mathbb{E}_{S\sim x}[F(S)].
\end{equation*}
Multilinear extensions are extensively used for maximizing their corresponding discrete submodular set functions and are known to be a special case of non-concave DR-submodular functions. The Hessian matrix of this class of functions has non-positive off-diagonal entries and all its diagonal entries are zero. 
It has been shown that multilinear extensions can be efficiently computed for a large class of submodular set functions, for example, weighted matroid rank function, set cover function, probabilistic coverage function and graph cut function 
(see \citet{iyerMonotoneClosureRelaxeda} for more examples and details).
\paragraph{\textbf{Indefinite quadratic functions.}} Consider the quadratic function $f(x)=\frac{1}{2}x^T Hx+h^T x+c$. If the matrix $H$ is element-wise non-positive, $f$ is a DR-submodular function.\\
More generally, if $h_i:\R \to \R$ is concave for all $i\in[n]$ and $\theta_{ij}\leq0~\forall i\neq j$, the following function $f:\R_+^n \to \R$ is DR-submodular:
\begin{equation*}
f(x)=\sum_{i=1}^n h_i (x_i)+\sum_{i,j:i\neq j}\theta_{ij}x_i x_j.
\end{equation*}
\paragraph{\textbf{Log-determinant function.}} Let the function $f:[0,1]^n\to \R$ be defined as
\begin{equation*}
	f(x)=\log\det \big({\rm diag}(x)(L-I)+I\big),
\end{equation*} 
where $L\succeq 0$ is a positive semidefinite matrix and ${\rm diag}(x)$ denotes a diagonal matrix with vector $x$ on its diagonal. This function is used as the objective function in Determinantal Point Processes (DPPs). It was proved in \citet{gillenwaterNearOptimalMAPInference} that $f$ is a DR-submodular function.
\section{Motivating applications}\label{apps}
In the following, in order to illustrate the generality of our framework, we have listed a number of interesting applications that could be cast into our setting.
\paragraph{\textbf{Online ad allocation.}} Consider the following online ad placement problem: At round $t\in[T]$, an advertiser should choose an investment vector $x_t \in R_+^n$ over $n$ different websites where $i$-th entry of $x_t$ denotes the amount that the advertiser is willing to pay per each click on the ad on the $i$-th website (i.e., cost per click). In other words, each website has different tiers of ads and choosing $x_t$ corresponds to ordering a certain type of ad. The aggregate cost of investment is determined when the number of clicks the ad receives is revealed. Namely, the cost of such an investment is characterized by $p_t$ where the $i$-th entry of the vector $p_t$ is the number of clicks the ad on the $i$-th website receives. The stochastic nature of the number of visitors of these $n$ websites validates our choice of stochastic linear constraint functions. The advertiser needs to balance her total investment against an allotted long-term budget $B_T$. At round $t\in[T]$, the advertiser's utility function $f_t (x_t)$ is a monotone DR-submodular function with respect to the vector of investments and this function quantifies the overall impressions of the ads. DR-submodularity of the utility function characterizes the diminishing returns property of the impressions. In other words, making an ad more visible will attract proportionally fewer extra viewers because each website shares a portion of its visitors with other websites.

\paragraph{\textbf{Online task assignment in crowdsourcing markets.}} In this problem, there exists a requester with a limited budget $B_T$ that submit jobs and benefits from them being completed. There are $n$ types of jobs available to be assigned to workers arriving online. At each step $t\in[T]$, a worker arrives and the requester has to assign a bundle $x_t \in \X=\{x\in \R_+^n:0\preceq x \preceq 1\}$ of the jobs to the worker. The worker has a cost $[p_t]_i~\forall i\in[n]$ for performing one unit of the $i$-th job where $[p_t]_i$ denotes the $i$-th entry of vector $p_t$. The workers' evaluation of the cost of performing each of these $n$ jobs is governed by the fluctuations of the wages in the job market and is stochastic in nature. The rewards obtained by the requester from this job assignment is a DR-submodular function $f_t(x_t)$. The DR property of the utility function captures the diminishing returns of assigning more jobs to the worker, i.e., as the number of assigned jobs to the worker increases, she has less time to devote to each fixed job $i\in[n]$ and therefore, the reward (quality of the completed task) obtained from the worker performing one unit of job $i$ decreases. In other words, if $x\preceq y$, $\nabla_i f(x)\geq \nabla_i f(y)~\forall i\in[n]$ holds. The goal is to maximize the overall rewards obtained by the requester while the budget constraint is not violated as well. Note that if the jobs are indivisible, for all $t\in[T]$, the utility function $f_t$ corresponds to the multilinear extension of the monotone submodular set function $F_t:2^n \to \R$ and using the lossless pipage rounding technique of \citet{calinescuMaximizingMonotoneSubmodular2011}, we allocate an integral bundle of jobs to the workers at each step.
\paragraph{\textbf{Online welfare maximization with production cost \citep{HUANG2019648}}} In this problem, there are $n$ types of products for sale that may be produced on demand using a fixed limited budget $B_T$. At each step $t\in[T]$, an agent (customer) arrives online and the algorithm has to assign a bundle $x_t \in \X=\{x\in \R_+^n:0\preceq x \preceq 1\}$ of products to the agent. Producing each unit of product $i\in[n]$ costs an unknown amount $[p_t]_i$ and the production cost of the item may change over time $\{1,\dots,T\}$ because of the stochastic fluctuations of the prices of ingredients. The agent has an unknown private DR-submodular valuation function $f_t$ over the items where the DR property characterizes the diversity of the assigned bundle. Therefore, the utility obtained by assigning the bundle $x_t$ equals $f_t(x_t)$. The goal is to maximize the overall valuation of the agents while satisfying the budget constraint. Note that if the products are indivisible, for all $t\in[T]$, the utility function $f_t$ corresponds to the multilinear extension of the monotone submodular set function $F_t:2^n \to \R$ and using the lossless pipage rounding technique of \citet{calinescuMaximizingMonotoneSubmodular2011}, we allocate an integral bundle of products to the agents at each step.  
\section{Proof of Lemma 1}\label{Appendix A}
Fix $k\in[K]$. Using $L$-smoothness of the function $\L_t$, we have:
\begin{align*}
&\L_t (x_t^{(k+1)},\lambda_t)\geq\\ &\L_t(x_t^{(k)},\lambda_t)+\frac{1}{K}\langle \nabla_x \L_t (x_t^{(k)},\lambda_t),v_t^{(k)}\rangle-\frac{L}{2K^2}\|v_t^{(k)}\|_2^2\\
&\overset{\text{(a)}}\geq \L_t(x_t^{(k)},\lambda_t)+\frac{1}{K}\langle \nabla_x \L_t (x_t^{(k)},\lambda_t),v_t^{(k)}\rangle-\frac{LR^2}{2K^2}\\
&= \L_t(x_t^{(k)},\lambda_t)+\frac{1}{K}\langle \nabla_x \L_t (x_t^{(k)},\lambda_t),v_t^{(k)}-x\rangle\\
&+\frac{1}{K}\langle \nabla_x \L_t (x_t^{(k)},\lambda_t),x\rangle-\frac{LR^2}{2K^2}\\
&= \L_t(x_t^{(k)},\lambda_t)+\frac{1}{K}\langle \nabla_x \L_t (x_t^{(k)},\lambda_t),v_t^{(k)}-x\rangle\\
&+\frac{1}{K}\langle \nabla f_t (x_t^{(k)}),x\rangle-\frac{1}{K}\lambda_t \langle \nabla \widehat{g}_t(x_t^{(k)}),x\rangle-\frac{LR^2}{2K^2}\\
&\overset{\text{(b)}}= \L_t(x_t^{(k)},\lambda_t)+\frac{1}{K}\langle \nabla_x \L_t (x_t^{(k)},\lambda_t),v_t^{(k)}-x\rangle\\
&+\frac{1}{K}\langle \nabla f_t (x_t^{(k)}),x\rangle-\frac{1}{K}\lambda_t \widehat{g}_t (x)-\frac{1}{K} \lambda_t \frac{B_T}{T}-\frac{LR^2}{2K^2},
\end{align*} 
where (a) is due to the assumption that ${\rm diam}(\X)\leq R$. Note that in order to obtain (b), we have used linearity of the budget functions for all $t\in[T]$ to write $\langle \nabla \widehat{g}_t(x_t^{(k)}),x\rangle=\langle \widehat{p}_t,x\rangle=\widehat{g}_t (x)+\frac{B_T}{T}$. More general assumptions such as convexity would not be enough for the proof to go through.\\
Considering that $f_t(x)$ is monotone DR-submodular for all $t\in[T]$, we can write:
\begin{align*}
f_t (x)-f_t(x_t^{(k)}) &\overset{\text{(c)}}\leq f_t (x\vee x_t^{(k)})-f_t(x_t^{(k)})\\
&\overset{\text{(d)}}\leq \langle \nabla f_t (x_t^{(k)}),(x\vee x_t^{(k)})-x_t^{(k)}\rangle\\
&= \langle \nabla f_t(x_t^{(k)}),(x-x_t^{(k)})\vee 0\rangle\\
&\overset{\text{(e)}}\leq \langle \nabla f_t(x_t^{(k)}),x\rangle,
\end{align*}
where for $a,b\in \R^n$, $a\vee b$ denotes the entry-wise maximum of vectors $a$ and $b$, (c) and (e) are due to monotonocity of $f_t$ and (d) uses concavity of $f_t$ along non-negative directions.\\
Therefore, we conclude:
\begin{align*}
\L_t (x_t^{(k+1)},\lambda_t)&\geq \L_t (x_t^{(k)},\lambda_t)\\ &+\frac{1}{K}\langle \nabla_x \L_t (x_t^{(k)},\lambda_t),v_t^{(k)}-x\rangle\\
&+\frac{1}{K}\big(f_t (x)-f_t(x_t^{(k)})\big)-\frac{1}{K}\lambda_t \widehat{g}_t (x)\\
&-\frac{1}{K} \lambda_t \frac{B_T}{T}-\frac{LR^2}{2K^2}.
\end{align*}
Equivalently, we can write:
\begin{align}
&f_t(x)-f_t(x_t^{(k+1)})\leq (1-\frac{1}{K})\big(f_t(x)-f_t(x_t^{(k)})\big)\nonumber\\
&-\lambda_t \big(\widehat{g}_t(x_t^{(k+1)})-\widehat{g}_t(x_t^{(k)})\big)+\frac{1}{K}\lambda_t \widehat{g}_t(x)\nonumber+\frac{1}{K} \lambda_t \frac{B_T}{T}\nonumber\\
&+\frac{LR^2}{2K^2}+\frac{1}{K}\langle \nabla \L_t(x_t^{(k)},\lambda_t),x -v_t^{(k)}\rangle\nonumber\\
&=(1-\frac{1}{K})\big(f_t(x)-f_t(x_t^{(k)})\big)+\frac{1}{K} \big[\lambda_t \frac{B_T}{T}-\lambda_t \langle \widehat{p}_t, v_t^{(k)}\rangle \nonumber\\
&+\lambda_t \widehat{g}_t(x)+\frac{LR^2}{2K}+\langle \nabla \L_t(x_t^{(k)},\lambda_t),x -v_t^{(k)}\rangle\big]\nonumber.
\end{align}
Taking the sum over $t\in\{1,\dots,T\}$, we obtain:
\begin{align}
&\sum_{t=1}^{T}\big(f_{t}(x)-f_{t}(x_{t}^{(k+1)})\big)\leq\nonumber\\ &(1-\frac{1}{K})\sum_{t=1}^{T}\big(f_{t}(x)-f_{t}(x_{t}^{(k)})\big)\nonumber\\
&+\frac{1}{K} \sum_{t=1}^{T}\big[-\lambda_{t} \langle \widehat{p}_{t}, v_{t}^{(k)}\rangle +\lambda_{t} \widehat{g}_{t}(x)\nonumber\\
&+\lambda_{t} \frac{B_T}{T}+\frac{LR^2}{2K}+\langle \nabla \L_{t}(x_{t}^{(k)},\lambda_{t}),x -v_{t}^{(k)}\rangle\big]\label{reg44_3}.
\end{align}
Applying inequality ($\ref{reg44_3}$) recursively for all $k\in\{1,\dots,K\}$, we obtain:
\begin{align}
&\sum_{t=1}^{T}\big(f_{t}(x)-f_{t}(\underbrace{x_{t}^{(K+1)}}_{=x_{t}})\big)\leq\nonumber\\ &\Pi_{k=0}^{K-1}(1-\frac{1}{K})\sum_{t=1}^{T}\big(f_{t}(x)-f_{t}(x_{t}^{(0)})\big)\nonumber\\
&+\sum_{k=0}^{K-1}\frac{1}{K} \Pi_{j=k+1}^{K-1}(1-\frac{1}{K}) \sum_{t=1}^{T}\big[-\lambda_{t} \langle \widehat{p}_{t}, v_{t}^{(k)}\rangle \nonumber\\
&+\lambda_{t} \widehat{g}_{t}(x)+\lambda_{t} \frac{B_T}{T}+\frac{LR^2}{2K}\nonumber\\
&+\langle \nabla \L_{t}(x_{t}^{(k)},\lambda_{t}),x -v_{t}^{(k)}\rangle\big]\label{reg44_4}.
\end{align}
Using the regret bound of Online Gradient Ascent instance $\Eps_k~\forall k\in[K]$, the following holds:
\begin{align*}
&\sum_{t=1}^T \langle \nabla_x \L_t (x_t^{(k)},\lambda_t),x-v_t^{(k)}\rangle=\\
&\sum_{t=1}^T \langle \nabla_x \L_t (x_t^{(k)},\lambda_t),x\rangle -\sum_{t=1}^T \langle \nabla_x \L_t (x_t^{(k)},\lambda_t),v_t^{(k)}\rangle\\
&\leq \frac{R^2}{\mu}+\frac{\mu}{2}\sum_{t=1}^T \|\nabla_x \L_t (x_t^{(k)},\lambda_t)\|^2\\
&=\frac{R^2}{\mu}+\frac{\mu}{2}\sum_{t=1}^T \|\nabla_x f_t (x_t^{(k)})-\lambda_t \widehat{p}_t\|^2\\
&\overset{\text{(a)}}\leq \frac{R^2}{\mu}+\frac{\mu}{2}\sum_{t=1}^T \big(2\|\nabla_x f_t (x_t^{(k)})\|^2+2\lambda_t^2 \|\widehat{p}_t\|^2\big)\\
&\overset{\text{(b)}}\leq \frac{R^2}{\mu}+\beta^2 \mu T+ \beta^2 \mu \sum_{t=1}^T\lambda_t^2,
\end{align*}
where (a) uses the inequality $\|a+b\|^2\leq 2\|a\|^2+2\|b\|^2~\forall a,b\in \R^n$ and (b) is due to $\beta$-Lipschitzness of functions $f_t, g_t$ for all $t\in[T]$.\\
Using the inequality $(1-\frac{1}{K})^K \leq \frac{1}{e}$ in ($\ref{reg44_4}$), we have:
\begin{align}
&\sum_{t=1}^{T}\big(f_{t}(x)-f_{t}(x_{t})\big)\leq\nonumber\\ &\frac{1}{e}\sum_{t=1}^{T}\big(f_{t}(x)-f_{t}(x_{t}^{(0)})\big)\nonumber\\
&+\sum_{t=1}^{T}\sum_{k=0}^{K-1}\frac{1}{K}\big[-\lambda_{t} \langle \widehat{p}_{t}, v_{t}^{(k)}\rangle +\lambda_{t} \widehat{g}_{t}(x)\nonumber\\
&+\lambda_{t} \frac{B_T}{T}+\frac{LR^2}{2K}+\langle \nabla \L_{t}(x_{t}^{(k)},\lambda_{t}),x -v_{t}^{(k)}\rangle\big]\nonumber\\
&=\frac{1}{e}\sum_{t=1}^{T}\big(f_{t}(x_t^*)-\underbrace{f_{t}(0)}_{=0}\big)\nonumber\\
&+\sum_{t=1}^{T}\big[-\lambda_{t}\widehat{g}_{t}(x_{t})-\lambda_{t}\frac{B_T}{T} +\lambda_{t} \widehat{g}_{t}(x)\nonumber\\
&+\lambda_{t} \frac{B_T}{T}+\frac{LR^2}{2K}+\sum_{k=0}^{K-1}\frac{1}{K}\langle \nabla \L_{t}(x_{t}^{(k)},\lambda_{t}),x -v_{t}^{(k)}\rangle\big]\label{reg44_5}.
\end{align}
Rearranging the terms in ($\ref{reg44_5}$), we obtain:
\begin{align}
&\sum_{t=1}^{T}\big((1-\frac{1}{e})f_{t}(x)-f_{t}(x_{t})\big)+\sum_{t=1}^{T}\lambda_{t}\widehat{g}_{t}(x_{t})\nonumber\\ & \leq \frac{LR^2T}{2K}+\sum_{t=1}^{T}\lambda_{t} \widehat{g}_{t}(x)\nonumber \\& +\sum_{k=0}^{K-1}\frac{1}{K}\sum_{t=1}^{T}\langle \nabla \mathcal{L}_{t}(x_{t}^{(k)},\lambda_{t}),x -v_{t}^{(k)}\rangle\nonumber\\
&\leq \frac{LR^2T}{2K}+\sum_{t=1}^{T}\lambda_{t}\widehat{g}_{t}(x)+\frac{R^2}{\mu}+\beta^2 \mu T+ \beta^2 \mu \sum_{t=1}^T\lambda_t^2\label{reg44_6}.
\end{align}
Now, subtract $\sum_{t=1}^{T}\lambda_t \gamma_t$ from each side of the inequality ($\ref{reg44_6}$). Thus, obtain that:
\begin{align}
&\sum_{t=1}^{T}\big((1-\frac{1}{e})f_{t}(x)-f_{t}(x_{t})\big) \leq \frac{LR^2T}{2K}+\frac{R^2}{\mu}+\beta^2 \mu T\nonumber\\ 
& + \sum_{t=1}^{T}\lambda_{t}\widetilde{g}_{t}(x)+ \big[\beta^2 \mu \sum_{t=1}^T\lambda_t^2 -\sum_{t=1}^{T}\lambda_{t}\widetilde{g}_{t}(x_{t}) \big]\nonumber.
\end{align}
Setting $\delta=\beta^2$ and using the update rule of the algorithm for $\lambda_t~\forall t\in[T]$, we have:
\begin{align}
&\beta^2 \mu \sum_{t=1}^T \lambda_t^2 -\sum_{t=1}^T \lambda_t \widetilde{g}_t(x_t) \nonumber \\
&=\frac{\beta^2 \mu}{\delta^2 \mu^2}\sum_{t=1}^T [\widetilde{g}_t(x_t)]_+^2-\frac{1}{\delta \mu}\sum_{t=1}^T [\widetilde{g}_t(x_t)]_+ \widetilde{g}_t(x_t)\nonumber\\
&\overset{\text{(a)}}\leq \frac{1}{\delta \mu}\sum_{t=1}^T \big(\widetilde{g}_t^2(x_t)- \widetilde{g}_t(x_t)\widetilde{g}_t(x_t))\nonumber\\
&\leq 0\nonumber,
\end{align}
where we have used $[\widetilde{g}_t(x_t)]_+ \geq \widetilde{g}_t(x_t)$ and $|[\widetilde{g}_t(x_t)]_+|\leq \widetilde{g}_t(x_t)$ to obtain (a).\\
\section*{Modified Online Lagrangian Frank-Wolfe (MOLFW) Algorithm}
For the setting where only unbiased stochastic gradient estimates of the utility functions $\{f_t\}_{t=1}^T$ with bounded variance $\sigma^2$ are available, inspired by the variance reduction technique introduced by \citet{pmlr-v80-chen18c}, we propose the Modified Online Lagrangian Frank-Wolfe (MOLFW) algorithm in Algorithm 2. Compared to the OLFW algorithm, the MOLFW algorithm uses $\langle v_t^{(k)},d_t^{(k)}\rangle$ (instead of $\langle v_t^{(k)},\nabla_x \L_t (x_t^{(k)},\lambda_t)\rangle$) as the payoff to be received by $\mathcal{E}_k~\forall k\in[K]$ at round $t\in[T]$. We analyze the performance of the MOLFW algorithm below.\\
Fix $k\in[K]$. Using $L$-smoothness of the function $\L_t$, we have:
\begin{align*}
&\L_t (x_t^{(k+1)},\lambda_t)\geq\\ &\L_t(x_t^{(k)},\lambda_t)+\frac{1}{K}\langle \nabla_x \L_t (x_t^{(k)},\lambda_t),v_t^{(k)}\rangle-\frac{L}{2K^2}\|v_t^{(k)}\|_2^2\\
&\overset{\text{(a)}}\geq \L_t(x_t^{(k)},\lambda_t)+\frac{1}{K}\langle \nabla_x \L_t (x_t^{(k)},\lambda_t),v_t^{(k)}\rangle-\frac{LR^2}{2K^2}\\
&= \L_t(x_t^{(k)},\lambda_t)+\frac{1}{K}\langle \nabla_x \L_t (x_t^{(k)},\lambda_t)-d_t^{(k)},v_t^{(k)}-x\rangle\\
&+\frac{1}{K}\langle \nabla_x \L_t (x_t^{(k)},\lambda_t),x\rangle+\frac{1}{K}\langle d_t^{(k)},v_t^{(k)}-x\rangle-\frac{LR^2}{2K^2}\\
&= \L_t(x_t^{(k)},\lambda_t)+\frac{1}{K}\langle \nabla_x \L_t (x_t^{(k)},\lambda_t)-d_t^{(k)},v_t^{(k)}-x\rangle\\
&+\frac{1}{K}\langle \nabla f_t(x_t^{(k)}),x\rangle-\frac{1}{K}\lambda_t\langle \widehat{p}_t,x\rangle+\frac{1}{K}\langle d_t^{(k)},v_t^{(k)}-x\rangle\\
&-\frac{LR^2}{2K^2},
\end{align*} 
where (a) is due to the assumption that ${\rm diam}(\X)\leq R$. Considering that $f_t(x)$ is monotone DR-submodular for all $t\in[T]$, we can write:
\begin{align*}
f_t (x)-f_t(x_t^{(k)}) &\overset{\text{(b)}}\leq f_t (x\vee x_t^{(k)})-f_t(x_t^{(k)})\\
&\overset{\text{(c)}}\leq \langle \nabla f_t (x_t^{(k)}),(x\vee x_t^{(k)})-x_t^{(k)}\rangle\\
&= \langle \nabla f_t(x_t^{(k)}),(x-x_t^{(k)})\vee 0\rangle\\
&\overset{\text{(d)}}\leq \langle \nabla f_t(x_t^{(k)}),x\rangle,
\end{align*}
where for $a,b\in \R^n$, $a\vee b$ denotes the entry-wise maximum of vectors $a$ and $b$, (b) and (d) are due to monotonocity of $f_t$ and (c) uses concavity of $f_t$ along non-negative directions.\\
Therefore, we conclude:
\begin{align*}
\L_t (x_t^{(k+1)},\lambda_t)&\geq \L_t (x_t^{(k)},\lambda_t)\\ &+\frac{1}{K}\big(\langle \nabla_x \L_t (x_t^{(k)},\lambda_t)-d_t^{(k)},v_t^{(k)}-x\rangle\\
&+f_t (x)-f_t(x_t^{(k)})-\lambda_t \widehat{g}_t (x)\\
&-\lambda_t \frac{B_T}{T}+\langle d_t^{(k)},v_t^{(k)}-x \rangle\big)-\frac{LR^2}{2K^2}.
\end{align*}
Using the Young's inequality, we have:
\begin{align*}
    &\langle \nabla_x \L_t (x_t^{(k)},\lambda_t)-d_t^{(k)},v_t^{(k)}-x\rangle \geq\\ &-\frac{1}{2\beta_k}\|\nabla_x \L_t (x_t^{(k)},\lambda_t)-d_t^{(k)}\|^2-\frac{\beta_k}{2}\|v_t^{(k)}-x\|^2 \geq\\
    &-\frac{1}{2\beta_k}\|\nabla_x \L_t (x_t^{(k)},\lambda_t)-d_t^{(k)}\|^2-\frac{R^2\beta_k}{2}.
\end{align*}
Therefore, we can equivalently write:
\begin{align}
&f_t(x)-f_t(x_t^{(k+1)})\leq (1-\frac{1}{K})\big(f_t(x)-f_t(x_t^{(k)})\big)\nonumber\\
&-\lambda_t \big(\widehat{g}_t(x_t^{(k+1)})-\widehat{g}_t(x_t^{(k)})\big)+\frac{1}{K}\lambda_t \widehat{g}_t(x)\nonumber+\frac{1}{K} \lambda_t \frac{B_T}{T}\nonumber\\
&+\frac{LR^2}{2K^2}+\frac{1}{K}\langle d_t^{(k)},x -v_t^{(k)}\rangle+\frac{1}{K}\frac{R^2 \beta_k}{2}\nonumber\\
&+\frac{1}{K}\frac{1}{2\beta_k}\|\nabla_x \L_t (x_t^{(k)},\lambda_t)-d_t^{(k)}\|^2\nonumber\\
&=(1-\frac{1}{K})\big(f_t(x)-f_t(x_t^{(k)})\big)+\frac{1}{K} \big[\lambda_t \frac{B_T}{T}-\lambda_t \langle \widehat{p}_t, v_t^{(k)}\rangle \nonumber\\
&+\lambda_t \widehat{g}_t(x)+\frac{LR^2}{2K}+\langle d_t^{(k)},x -v_t^{(k)}\rangle+\frac{R^2\beta_k}{2}\nonumber\\
&+\frac{1}{2\beta_k}\|\nabla_x \L_t (x_t^{(k)},\lambda_t)-d_t^{(k)}\|^2\big]\nonumber.
\end{align}
Taking the sum over $t\in\{1,\dots,T\}$, we obtain:
\begin{align}
&\sum_{t=1}^{T}\big(f_{t}(x)-f_{t}(x_{t}^{(k+1)})\big)\leq\nonumber\\ &(1-\frac{1}{K})\sum_{t=1}^{T}\big(f_{t}(x)-f_{t}(x_{t}^{(k)})\big)\nonumber\\
&+\frac{1}{K} \sum_{t=1}^{T}\big[-\lambda_{t} \langle \widehat{p}_{t}, v_{t}^{(k)}\rangle +\lambda_{t} \widehat{g}_{t}(x)\nonumber\\
&+\lambda_{t} \frac{B_T}{T}+\frac{LR^2}{2K}+\langle d_t^{(k)},x-v_t^{(k)}\rangle+\frac{R^2\beta_k}{2}\nonumber\\
&+\frac{1}{2\beta_k}\|\nabla_x \L_t (x_t^{(k)},\lambda_t)-d_t^{(k)}\|^2\big]\label{reg44_3s}.
\end{align}
\begin{algorithm}[ht]
	\caption{Modified Online Lagrangian Frank-Wolfe (MOLFW)}
	\begin{algorithmic}
		\STATE \textbf{Input:} $\X$ is the constraint set, $T$ is the horizon, $\mu>0$, $\delta>0$, $\{\gamma_t\}_{t=1}^{T}$, $\{\rho_k\}_{k=1}^K$ and $K$.
		\STATE \textbf{Output:} $\{x_t :1\leq t\leq T\}$.
		\STATE Initialize $K$ instances $\{\Eps_k\}_{k\in[K]}$ of Online Gradient Ascent with step size $\mu$ for online maximization of linear functions over $\mathcal{X}$.
		\FOR{$t=1$ {\bfseries to} $T$}
		\STATE $x_t^{(1)}=0$.
		\FOR{$k=1$ {\bfseries to} $K$}
		\STATE Let $v_t^{(k)}$ be the output of oracle $\mathcal{E}_k$ from round $t-1$.
		\STATE $x_t^{(k+1)}=x_t^{(k)}+\frac{1}{K} v_t^{(k)}$.
		\ENDFOR
		\STATE Set $x_t =x_t^{(K+1)}$.
		\STATE Let $\widehat{p}_t := \frac{1}{t-1}\sum_{s=1}^{t-1}p_s$ for $t >1$.
		\STATE Let \[ \widetilde{g}_t(\cdot)= 
		\begin{cases}
		\langle \widehat{p}_t, \cdot \rangle -\frac{B_T}{T} &\quad\text{expectation analysis (\rom{1})}\\
		\langle \widehat{p}_t, \cdot \rangle -\frac{B_T}{T} - \gamma_t &\quad\text{high probability analysis (\rom{2})}\\
		\end{cases}
		\] 
		\STATE Set $\lambda_{t}= \frac{[\widetilde{g}_t(x_t)]_+}{\delta \mu}$ for $t > 1$ and 0 otherwise.
		\STATE Play $x_t$ and observe the Lagrangian function $\L_t(x_t,\lambda_t)=f_t(x_t)-\lambda_t \widetilde{g}_t(x_t)+\frac{\delta \mu}{2}\lambda_t^2$
		\STATE $d_t^{(0)}=-\lambda_t \widehat{p}_t$.
		\FOR{$k=1$ {\bfseries to} $K$}
		\STATE $d_t^{(k)}=(1-\rho_k)d_t^{(k-1)}+\rho_k \nabla_x \L_t (x_t^{(k)},\lambda_t)$.
		\STATE Feedback $\langle v_t^{(k)},d_t^{(k)}\rangle$ as the payoff to be received by $\mathcal{E}_k$.
		\ENDFOR
		\ENDFOR
	\end{algorithmic}
\end{algorithm}
Applying inequality ($\ref{reg44_3s}$) recursively for all $k\in\{1,\dots,K\}$, we obtain:
\begin{align}
&\sum_{t=1}^{T}\big(f_{t}(x)-f_{t}(\underbrace{x_{t}^{(K+1)}}_{=x_{t}})\big)\leq\nonumber\\ &\Pi_{k=0}^{K-1}(1-\frac{1}{K})\sum_{t=1}^{T}\big(f_{t}(x)-f_{t}(x_{t}^{(0)})\big)\nonumber\\
&+\sum_{k=0}^{K-1}\frac{1}{K} \Pi_{j=k+1}^{K-1}(1-\frac{1}{K}) \sum_{t=1}^{T}\big[-\lambda_{t} \langle \widehat{p}_{t}, v_{t}^{(k)}\rangle \nonumber\\
&+\lambda_{t} \widehat{g}_{t}(x)+\lambda_{t} \frac{B_T}{T}+\frac{LR^2}{2K}\nonumber\\
&+\langle d_t^{(k)},x-v_t^{(k)}\rangle+\frac{R^2\beta_k}{2}\nonumber\\
&+\frac{1}{2\beta_k}\|\nabla_x \L_t (x_t^{(k)},\lambda_t)-d_t^{(k)}\|^2\big]\label{reg44_4s}.
\end{align}
Using the regret bound of Online Gradient Ascent instance $\Eps_k~\forall k\in[K]$, the following holds:
\begin{align*}
&\sum_{t=1}^T \langle d_t^{(k)},x-v_t^{(k)}\rangle=\\
&\sum_{t=1}^T \langle d_t^{(k)},x\rangle -\sum_{t=1}^T \langle d_t^{(k)},v_t^{(k)}\rangle\\
&\leq \frac{R^2}{\mu}+\frac{\mu}{2}\sum_{t=1}^T \|d_t^{(k)}\|^2\\
&\leq \frac{R^2}{\mu}+\mu\sum_{t=1}^T \|\nabla_x \L_t (x_t^{(k)},\lambda_t)-d_t^{(k)}\|^2\\
&+\mu\sum_{t=1}^T\|\nabla_x \L_t (x_t^{(k)},\lambda_t)\|^2\\
&\overset{\text{(a)}}\leq \frac{R^2}{\mu}+\mu\sum_{t=1}^T \|\nabla_x \L_t (x_t^{(k)},\lambda_t)-d_t^{(k)}\|^2\\
&+\mu\sum_{t=1}^T \big(2\|\nabla f_t (x_t^{(k)})\|^2+2\lambda_t^2 \|\widehat{p}_t\|^2\big)\\
&\overset{\text{(b)}}\leq \frac{R^2}{\mu}+2\beta^2 \mu T+ 2\beta^2 \mu \sum_{t=1}^T\lambda_t^2\\
&+\mu\sum_{t=1}^T \|\nabla_x \L_t (x_t^{(k)},\lambda_t)-d_t^{(k)}\|^2,
\end{align*}
where (a) uses the inequality $\|a+b\|^2\leq 2\|a\|^2+2\|b\|^2~\forall a,b\in \R^n$ and (b) is due to $\beta$-Lipschitzness of functions $f_t, g_t$ for all $t\in[T]$.\\
Using the inequality $(1-\frac{1}{K})^K \leq \frac{1}{e}$ in ($\ref{reg44_4s}$), we have:
\begin{align}
&\sum_{t=1}^{T}\big(f_{t}(x)-f_{t}(x_{t})\big)\leq\nonumber\\ &\frac{1}{e}\sum_{t=1}^{T}\big(f_{t}(x)-f_{t}(x_{t}^{(0)})\big)\nonumber\\
&+\sum_{t=1}^{T}\sum_{k=0}^{K-1}\frac{1}{K}\big[-\lambda_{t} \langle \widehat{p}_{t}, v_{t}^{(k)}\rangle +\lambda_{t} \widehat{g}_{t}(x)\nonumber\\
&+\lambda_{t} \frac{B_T}{T}+\frac{LR^2}{2K}+\langle d_t^{(k)},x-v_t^{(k)}\rangle+\frac{R^2\beta_k}{2}\nonumber\\
&+\frac{1}{2\beta_k}\|\nabla_x \L_t (x_t^{(k)},\lambda_t)-d_t^{(k)}\|^2\big]\nonumber\\
&=\frac{1}{e}\sum_{t=1}^{T}\big(f_{t}(x)-\underbrace{f_{t}(0)}_{=0}\big)\nonumber\\
&+\sum_{t=1}^{T}\big[-\lambda_{t}\widehat{g}_{t}(x_{t})-\lambda_{t}\frac{B_T}{T} +\lambda_{t} \widehat{g}_{t}(x)\nonumber\\
&+\lambda_{t} \frac{B_T}{T}+\frac{LR^2}{2K}+\sum_{k=0}^{K-1}\frac{1}{K}\big(\langle d_t^{(k)},x-v_t^{(k)}\rangle+\frac{R^2\beta_k}{2}\nonumber\\
&+\frac{1}{2\beta_k}\|\nabla_x \L_t (x_t^{(k)},\lambda_t)-d_t^{(k)}\|^2\big)\big]\label{reg44_5s}.
\end{align}
Rearranging the terms in ($\ref{reg44_5s}$), we obtain:
\begin{align}
&\sum_{t=1}^{T}\big((1-\frac{1}{e})f_{t}(x)-f_{t}(x_{t})\big)+\sum_{t=1}^{T}\lambda_{t}\widehat{g}_{t}(x_{t})\nonumber\\ & \leq \frac{LR^2T}{2K}+\sum_{t=1}^{T}\lambda_{t} \widehat{g}_{t}(x)\nonumber \\
&+\sum_{k=0}^{K-1}\frac{1}{K}\big(\langle d_t^{(k)},x-v_t^{(k)}\rangle+\frac{R^2\beta_k}{2}\nonumber\\
&+\frac{1}{2\beta_k}\|\nabla_x \L_t (x_t^{(k)},\lambda_t)-d_t^{(k)}\|^2\big)\nonumber\\
&\leq \frac{LR^2T}{2K}+\sum_{t=1}^{T}\lambda_{t}\widehat{g}_{t}(x)+\frac{R^2}{\mu}+2\beta^2 \mu T+ 2\beta^2 \mu \sum_{t=1}^T\lambda_t^2\nonumber\\
&+\sum_{t=1}^T\sum_{k=0}^{K-1}\frac{1}{K}\big(\frac{R^2\beta_k}{2}+\frac{1}{2\beta_k}\|\nabla_x \L_t (x_t^{(k)},\lambda_t)-d_t^{(k)}\|^2\nonumber\\
&+\mu\sum_{t=1}^T \|\nabla_x \L_t (x_t^{(k)},\lambda_t)-d_t^{(k)}\|^2\big)\label{reg44_6s}.
\end{align}
Now, subtract $\sum_{t=1}^{T}\lambda_t \gamma_t$ from each side of the inequality ($\ref{reg44_6s}$). Thus, we obtain:
\begin{align}
&\sum_{t=1}^{T}\big((1-\frac{1}{e})f_{t}(x)-f_{t}(x_{t})\big)\nonumber\\ 
&\leq \frac{LR^2T}{2K}+\sum_{t=1}^{T}\lambda_{t}\widetilde{g}_{t}(x)+\frac{R^2}{\mu}+2\beta^2 \mu T\nonumber\\
&+\big[2\beta^2 \mu \sum_{t=1}^T\lambda_t^2-\sum_{t=1}^T\lambda_t\widetilde{g}_t(x_t)\big]\nonumber\\
&+\sum_{t=1}^T\sum_{k=0}^{K-1}\frac{1}{K}\big(\frac{R^2\beta_k}{2}+\frac{1}{2\beta_k}\|\nabla_x \L_t (x_t^{(k)},\lambda_t)-d_t^{(k)}\|^2\nonumber\\
&+\mu\sum_{t=1}^T \|\nabla_x \L_t (x_t^{(k)},\lambda_t)-d_t^{(k)}\|^2\big)\nonumber.
\end{align}
Setting $\delta=2\beta^2$ and using the update rule of the algorithm for $\lambda_t~\forall t\in[T]$, we have:
\begin{align}
&2\beta^2 \mu \sum_{t=1}^T \lambda_t^2 -\sum_{t=1}^T \lambda_t \widetilde{g}_t(x_t) \nonumber\\
&=\frac{2\beta^2 \mu}{\delta^2 \mu^2}\sum_{t=1}^T [\widetilde{g}_t(x_t)]_+^2-\frac{1}{\delta \mu}\sum_{t=1}^T [\widetilde{g}_t(x_t)]_+ \widetilde{g}_t(x_t)\nonumber\\
&\overset{\text{(a)}}\leq \frac{1}{\delta \mu}\sum_{t=1}^T \big(\widetilde{g}_t^2(x_t)- \widetilde{g}_t(x_t)\widetilde{g}_t(x_t))\nonumber\\
&\leq 0\nonumber,
\end{align}
where we have used $[\widetilde{g}_t(x_t)]_+ \geq \widetilde{g}_t(x_t)$ and $|[\widetilde{g}_t(x_t)]_+|\leq \widetilde{g}_t(x_t)$ to obtain (a).\\
Using Theorem 3 of \citet{pmlr-v80-chen18c}, we have:
\begin{equation*}
    \E\big[\|\nabla_x \L_t (x_t^{(k)},\lambda_t)-d_t^{(k)}\|^2\big]\leq \frac{Q}{(k+4)^{2/3}}~\forall t\in[T],k\in[K],
\end{equation*}
where $Q=\max\{\max_{t\in[T]}\|\nabla f_t(0)\|^24^{2/3},4\sigma^2+6L^2R^2\}$. Therefore, taking expectation of both sides, setting $\beta_k=\frac{Q^{1/2}}{R(k+4)^{1/3}}$ and simplifying the result, we obtain:
\begin{align*}
 \E[R_T]&\leq \frac{R^2}{\mu}+2\beta^2\mu T+\frac{LR^2T}{2K}+\sum_{t=1}^T\E[\lambda_t \widetilde{g}_t(x)]+\frac{3Q\mu T}{K^{2/3}}\\
 &+\frac{3RQ^{1/2}T}{2K^{1/3}}.
\end{align*}
Therefore, if we set $K=\mathcal{O}(T^{3/2})$, we can analyze the expected regret and constraint violation of the MOLFW algorithm in a similar fashion to Theorem 3 and Theorem 4, and we obtain similar performance bounds. 
\section{Lemma \ref{lem4.2}}
\begin{lemma}\label{lem4.2}\citep{jinShortNoteConcentration2019}
	For all $t \in [T]$, the following holds: 
	\begin{align*}
	\P\{\|p_t - p\|_2 \geq \zeta\} \leq 2e^{-\frac{\zeta^2}{2\sigma^2}} \hspace{3mm}\forall \zeta \in \R.
	\end{align*}
\end{lemma}
\begin{proof}
From assumption A3, we have that $\|p_t\| \leq \beta$. Thus, Lemma 1 of \citet{jinShortNoteConcentration2019} holds with norm sub-gaussian parameter $\sigma = c\beta$ for some universal constant $c$. The result follows immediately. \qed
\end{proof}
\section{Lemma \ref{lem4.3}}
\begin{lemma}\label{lem4.3}
	For $t = 2, 3, \ldots, T$, the following holds with probability at least $1-\frac{\eps}{T}$:
	\begin{align*}
	\|\widehat{p}_t - p\| \leq c'\sigma\sqrt{\frac{\log(\frac{2nT}{\eps})}{t-1}}.
	\end{align*}
\end{lemma}
\begin{proof}
Note that $\|\widehat{p}_t - p\| = \frac{1}{t-1}\|\sum_{s=1}^{t-1}(p_s - p)\|$ and since vector $p_s - p$ satisfies the result of Lemma \ref{lem4.2}, we can apply Corollary 7 of \citet{jinShortNoteConcentration2019} to the random vectors $\{p_s - p\}_{s=1}^{t-1}$ and obtain:
\begin{align*}
\|\sum_{s=1}^{t-1}(p_s - p)\| &\leq c'\sqrt{\sum_{s=1}^{t-1}\sigma^2\log(\frac{2nT}{\eps})}\\
&= \sqrt{t-1}c'\sigma\sqrt{\log(\frac{2nT}{\eps})}.
\end{align*} 
Combining the above equations, we get the desired result. \qed
\end{proof}
\section{Proof of Lemma 2}
Consider the events $\Eps_t : =\{\|\widehat{p}_t - p\| \leq c'\sigma\sqrt{\log(\frac{2nT}{\eps})/(t-1)} \}$ for $t \in [T]/\{1\}$ and $\Eps := \{\sum_{t=2}^{T}\|\widehat{p}_t - p\| \leq C\sigma\sqrt{T\log(2nT/\eps)}\}$.  Then, apply union bound to $\bigcup_{t=2}^{T}\Eps_t^c$ with the observation that $\bigcap_{t=2}^{T}\Eps_t \subset \Eps$. Explicitly,
\begin{align*}
\P(\text{$\bigcup_{t=2}^{T}\Eps_t$}) & \leq \sum_{t=2}^{T}\P(\Eps_t^c),\\
1 - \P(\bigcap_{t=2}^{T}\Eps_t) &\leq \sum_{t=2}^{T}(1 - \P(\Eps_t)),\\
1 - \sum_{t=2}^{T}(1 - \P(\Eps_t)) & \leq \P(\bigcap_{t=2}^{T}\Eps_t) \leq \P(\Eps).
\end{align*}
Using the result of Lemma \ref{lem4.3}, we obtain the result.
\section{Proof of Lemma 4}
Bounding $\sum_{t=1}^{T}[\widetilde{g}_t(x_t)]_+ $ from below, we obtain:
\begin{align*}
\sum_{t=1}^{T}[\widetilde{g}_t(x_t)]_+ &\geq \sum_{t=1}^{T}\widetilde{g}_t(x_t)\\
&= \sum_{t=1}^{T}g(x_t) + \sum_{t=1}^{T}(\widehat{g}_t(x_t) - g(x_t)) - \sum_{t=1}^{T}\gamma_t\\
&\geq \sum_{t=1}^{T}g(x_t) - \sum_{t=1}^{T}|\widehat{g}_t(x_t) - g(x_t)| - \sum_{t=1}^{T}\gamma_t\\
&= C_T - \sum_{t=1}^{T}|\langle \widehat{p}_t-p, x_t\rangle| - \sum_{t=1}^{T}\gamma_t \\
&\geq C_T - \sum_{t=1}^{T}\|\widehat{p}_t - p\| \|x_t\|- \sum_{t=1}^{T}\gamma_t\\
&\geq C_T - R\sum_{t=1}^{T}\|\widehat{p}_t - p\| - \sum_{t=1}^{T}\gamma_t.
\end{align*}  
Rearranging the above inequality, we obtain the desired result.
\section{Proof of Lemma 5}\label{Appendix B}
We can write:
\begin{align*}
\E\|\widehat{p}_t-p\|^2 & = \E(\widehat{p}_t - p)^T(\widehat{p}_t-p) \nonumber \\ 
& = \E [Tr\big((\widehat{p}_t - p)^T(\widehat{p}_t-p)\big)] \nonumber \\
& = \E [Tr\big((\widehat{p}_t - p)(\widehat{p}_t-p)^T\big)] \nonumber \\
& = Tr\big(\E [(\widehat{p}_t - p)(\widehat{p}_t-p)^T]\big) \nonumber \\
& = Tr(Cov(\widehat{p}_t)) \nonumber \\
& = Tr\big(\frac{\Sigma}{t-1}\big) \nonumber \\
& = \frac{Tr(\Sigma)}{t-1}.
\end{align*}
\section{Proof of Theorem 4}
Similar to the proof of Theorem 2, we begin by setting $x=0$ to be the fixed vector in (2). We obtain 
\begin{align*}
\sum_{t=1}^T [\widetilde{g}_t(x_t)]_+ &\leq \frac{T}{B_T}R\beta F\sqrt{T} + \frac{TR\beta}{B_T}(\frac{LR^2}{2}+2R\beta).
\end{align*}
Now, we lower bound the left-hand side following the idea of the proof for Lemma 4. Thus, we obtain
\begin{align*}
C_T& \leq \frac{T}{B_T}R\beta F\sqrt{T} \nonumber +\frac{TR\beta}{B_T}(\frac{LR^2}{2}+2R\beta) + \underbrace{R\sum_{t=1}^{T}\|\widehat{p}_t-p\|}_{\text{(C)}}.
\end{align*}
In order to bound (C), we take expectation on both sides to obtain: 
\begin{align*}
\E[\text{(C)}] &=R \E[\sum_{t=2}^{T}\|\widehat{p}_t - p\|]\\
&= R\sum_{t=2}^{T}\E\|\widehat{p}_t - p\|\\
&= R\sum_{t=2}^{T}\E\sqrt{\| \widehat{p}_t - p\|^2}\\
&\leq R\sum_{t=2}^{T}\sqrt{\E\|\widehat{p}_t - p\|^2},
\end{align*}
where the last inequality is due to Jensen's inequality.\\
Therefore, we have:
\begin{align*}
\E[\text{(C)}] &\leq R\sum_{t=2}^{T}\frac{\sqrt{Tr(\Sigma)}}{\sqrt{t-1}}\\
&\leq R\sqrt{Tr(\Sigma)}\sqrt{T}. 
\end{align*}
Combining the inequalities, we obtain the result.

\end{document}